\documentclass[a4paper,11pt,reqno]{amsart}
\usepackage{amsmath}
\usepackage{amsthm,enumerate}

\usepackage{graphicx}
\usepackage{amssymb}

\usepackage{appendix}

\usepackage[italian,english]{babel}

\usepackage[utf8x]{inputenc}
\usepackage{fancyhdr}

\usepackage{calc}
\usepackage{url}

\usepackage[text={6in,8.6in},centering]{geometry}

\usepackage{srcltx, inputenc}

\usepackage[T1]{fontenc}

\usepackage[usenames,dvipsnames]{color}

\makeatletter
\def\cleardoublepage{\clearpage\if@twoside \ifodd\c@page\else%
         \hbox{}%
     \thispagestyle{empty}
     \newpage%
     \if@twocolumn\hbox{}\newpage\fi\fi\fi}
\makeatother

\let\cleardoublepage\clearpage

\addto\captionsitalian{}

\newtheorem{thm}{Theorem}[section]
\newtheorem{cor}[thm]{Corollary}
\newtheorem{lem}[thm]{Lemma}

\newtheorem{den}[thm]{Definition}
\newtheorem{oss}[thm]{Remark}
\numberwithin{equation}{section}

\begin{document}

\title[The PME with large data on general Cartan-Hadamard manifolds]{The porous medium equation with large data \\ on Cartan-Hadamard manifolds \\ under general curvature bounds}

\author {Gabriele Grillo} \author[Matteo Muratori]{Matteo Muratori} \author{Fabio Punzo}

\address {Gabriele Grillo, Matteo Muratori, and Fabio Punzo: Dipartimento di Matematica, Poli\-tecnico di Milano, Piaz\-za Leonardo da Vinci 32, 20133 Milano, Italy}
\email{gabriele.grillo@polimi.it}
\email{matteo.muratori@polimi.it}
\email{fabio.punzo@polimi.it}

%

\makeatletter
\@namedef{subjclassname@2020}{%
  \textup{2020} Mathematics Subject Classification}
\makeatother



%
%
%
%


\begin{abstract}
We consider very weak solutions of the Cauchy problem for the porous medium equation on Cartan-Hadamard manifolds, that are assumed to satisfy general curvature bounds and to be stochastically complete. We identify a class of initial data that can grow at infinity at a prescribed rate, which depends on the assumed curvature bounds through an integral function, such that the corresponding solution exists at least on $[0,T]$ for a suitable $T>0$. The maximal existence time $T$ is estimated in terms of a suitable weighted norm of the initial datum. Our results are sharp, in the sense that slower growth rates yield global existence, whereas one can construct data with critical growth for which the corresponding solutions blow up in finite time. Under further assumptions, uniqueness of very weak solutions is also proved, in the same growth class.
\end{abstract}

\maketitle

%




\section{Introduction}
We consider very weak (i.e.~distributional) solutions
of the following Cauchy problem:
\begin{equation}\label{e64}
\begin{cases}
u_t = \Delta \! \left( u^m \right) & \text{in } M\times (0,T)\,, \\
u = u_0 & \text{on } M \times \{ 0 \} \, ,
\end{cases}
\qquad m>1 \, ,
\end{equation}
where $M$ is a $N$-dimensional Riemannian manifold, supposed to be complete, simply connected, and to have nonpositive sectional curvature,  namely to be a \emph{Cartan-Hadamard manifold}. In the above differential equation, $\Delta$ is the Laplace-Beltrami operator on $M$, and since we assume that $m>1$ we are dealing with the so-called \emph{porous medium equation} on $M$. Note that, when considering sign-changing solutions, we implicitly define $ u^m := |u|^{m-1}u $, as it is customary in the literature and it guarantees the parabolicity of the problem. We are interested in allowing the initial datum $u_0$ to \emph{grow} and be \emph{unbounded} at infinity, so that in principle
the interval $[0, T)$ on which the solution exists may be finite and depend on $ u_0 $.

Our main goal will be to determine a rate that is \emph{critical} for the existence of solutions to \eqref{e64}, in the sense that if the datum grows with a strictly slower rate, then the corresponding solution is \emph{global} in time (i.e.~$ T=+\infty $), whereas one can construct examples in which the rate is critical and pointwise \emph{blow-up} occurs in finite time. We also aim at establishing a quantitative relation between the maximal existence time and a suitable norm of the initial datum that accounts for such a growth rate. Note that, in general, there may exist no solution at all to \eqref{e64} if one does not assume appropriate growth conditions on $u_0$ (see e.g.~\cite[Section 2.7]{GMP-pures}).

\smallskip

This kind of problems has been considered in full detail in the Euclidean case. In fact, B\'enilan, Crandall and Pierre were able in \cite{BCP} to identify, when $ M \equiv \mathbb R^N$, the optimal condition on the initial datum $u_0$ ensuring the existence of a solution. More precisely, it is shown there that if $u_0$ is locally integrable and satisfies
\begin{equation}\label{q52}
\sup_{R\geq 1} \frac 1{R^{N+\frac 2{m-1}}} \int_{B_R} \left| u_0(x) \right| dx  < + \infty \, ,
\end{equation}
then there exists a distributional solution $u$ of problem \eqref{e64}. Moreover, solutions corresponding to such initial data are proved to be unique in a suitable class. Besides, the set of data determined by \eqref{q52} is \emph{sharp} since Aronson and Caffarelli showed in \cite{AC} that initial traces of general solutions to the differential equation appearing in \eqref{e64} must necessarily fulfill \eqref{q52}.

In the Riemannian setting, there is a number of challenging additional difficulties, among which we single out the absence of a satisfactory analogue of the celebrated \it Aronson-B\'enilan inequality \rm (see \cite{AB}) for nonnegative solutions:
\[
\Delta\!\left(u^{m-1}\right) \ge -\frac Ct
\]
for a suitable $C=C(m,N)>0$, to be understood in the sense of distributions. While some Riemannian local analogues of such a kind of inequalities have been established in \cite{LNVV} and subsequently in \cite{HHL}, it does not seem straightforward to use such variants to prove well-posedness results in the spirit of \cite{BCP}. Nonetheless, closely related results in the case of Cartan-Hadamard manifolds have been investigated in \cite{GMP-pures} through different methods that, although not allowing to find full analogues of the integral condition \eqref{q52}, can be pushed to yield the correct growth rate at least in a pointwise sense. In fact, it is assumed in \cite{GMP-pures} that the (radial) Ricci curvature $\mathrm{Ric}_o$ with respect to a given pole $ o \in M $ satisfies the lower bound
\begin{equation}\label{power}
\mathrm{Ric}_o(x)\ge - C_0 \left[1+d(x,o)\right]^{\gamma} \qquad \forall x \in M \, ,
\end{equation}
for some constants $C_0>0$ and $ \gamma\in(-\infty,2)$, where we let $d(\cdot,\cdot)$ denote the geodesic distance on $ M $. The identified growth class of data depends explicitly on the value of the exponent $\gamma$: more precisely, existence (and uniqueness) of solutions is proved under the pointwise bound
\begin{equation}\label{power-growth}
\left| u_0(x) \right| \le C \left[ 1+d(x,o)\right]^{\frac\sigma{m-1}} \qquad \text{for a.e. } x \in M \, ,
\end{equation}
where $C>0$ is a suitable constant depending on $ u_0 $, and $\sigma:=\frac{2-\gamma}{2}\wedge 2$ (we set $ a \wedge b := \min\{a,b\} $). This condition is shown to be optimal in terms of global existence for a slower growth, blow-up phenomena for the critical growth and suitable nonexistence results for a faster growth. In the Euclidean case one can choose $ \gamma $ arbitrarily small, so that $\sigma=2$ and the growth class is in agreement with \eqref{q52}, whereas in the particularly significant case of the hyperbolic space $\mathbb{H}^N$ one has $ \gamma=0 $ and thus $\sigma=1$. Note that in the setting of \cite{GMP-pures}, it is important that the curvature satisfies a lower bound that is confined to be of \it power type\rm, and also that the critical case $\gamma=2$ was not addressed due to technical difficulties: while it is formally included in some of the stated results, the corresponding class of data according to \eqref{power-growth} turns out to be $L^\infty(M)$ (the constant $\sigma$  vanishes for $\gamma=2$), a space on which existence of solutions can be shown by standard methods, see the monograph \cite{Vaz07}. However, the $  L^\infty(M)$ class is not optimal: this was (partially) settled later in \cite{MP}.

The case of manifolds whose negative curvature diverges \emph{quadratically} as $d(x,o)$ tends to $+\infty$ is, in several senses, a threshold one. In the framework of nonlinear evolution equations, this can be seen by noticing significant differences between the long-time behavior
of solutions to the porous medium equation when $\gamma\le2$, see \cite{GMV1}, and the one emerging when $ \gamma>2 $ (superquadratic curvature), see \cite{GMV2}, but it is also apparent from the study of the (linear) heat equation. Indeed, we recall that a manifold is called \it stochastically complete \rm if
\begin{equation*}\label{stochcompl}
\int_M p(x,y,t) \, d\mu(y) =  1 \qquad \forall (x,t) \in M \times \mathbb{R}^+ \, ,
\end{equation*}
where $p(x,y,t)$ is the \emph{heat kernel} of $M$ and $d\mu$ is its Riemannian volume measure. Equivalently, a manifold is stochastically complete if the lifetime of the \emph{Brownian motion} acting on $M$ is almost surely $+\infty$. We refer e.g.~to \cite{Grig, Grig3} for geometric
 conditions on $M$ ensuring that stochastic completeness holds (see also Section \ref{lapl-comp}), but we comment that it does hold in the above setting if the radial Ricci curvature diverges \it at most quadratically\rm, namely \eqref{power} is satisfied for some $ \gamma \le 2 $. Informally, we may say that if the curvature becomes negative too fast at spatial infinity, then the Brownian motion accelerates so quickly that mass can be lost, still at spatial infinity. We also comment that stochastic completeness, a purely linear property, has recently been proved in \cite{GIM} to be equivalent to uniqueness of \it bounded \rm solutions to the \it fast diffusion equation\rm, namely to \eqref{e64} in the case $m<1$, and that, conversely, uniqueness fails even for bounded data if the manifold is stochastically incomplete. In particular, the quadratic divergence of the curvature to $-\infty$ turns out to be a threshold for that property as well.

\smallskip

Throughout we will thus require that $M$ is stochastically complete, or more precisely that it satisfies a lower curvature bound implying such a property, and provide improvements and extensions to the results of \cite{GMP-pures,MP}, that can be informally described as follows:

\begin{itemize}
\item the curvature bound \eqref{power} is replaced by a much more general bound, see \eqref{thmexi-ricci}, compatible with $M$ being stochastically complete, therefore considerably enlarging the class of manifolds dealt with;

\item the critical case $\gamma=2$ in \eqref{power} is included, providing existence and uniqueness (under further requirements) for data whose growth is of \emph{logarithmic} type.
\end{itemize}

Our main results are the following. First of all, we prove in Theorem \ref{thmexi} an existence result for solutions to \eqref{e64} on stochastically complete Cartan-Hadamard manifolds, satisfying a curvature bound from below which is much more general than \eqref{power}. The maximal existence time $T$ of solutions is estimated in terms of an appropriate weighted $ L^\infty $ norm of the initial datum, accounting for a precise critical growth rate, and in particular it is shown that data with a subcritical growth give rise to global solutions. We point out that such a rate is strictly related to the curvature bound by means of a special radial function $ H $ whose unboundedness is \emph{equivalent} to stochastic completeness on \emph{model manifolds}, a significant subclass of spherically symmetric Riemannian manifolds (see Sections \ref{RG}--\ref{lapl-comp}). For data having a critical growth, we exhibit examples of solutions, at least on model manifolds, that blow up pointwise everywhere in \emph{finite time}. This is the content of Theorem \ref{opt-blow}.

As concerns uniqueness, we prove three different results, under different additional assumptions on the manifold at hand or on the data. The first one, namely Theorem \ref{thm-modelli}, deals with model manifolds only. A further technical condition on the function $ H $ has to be required, see \eqref{ipotesiH-enunciato}. Nevertheless, it is worth observing that the latter is satisfied in the critical quadratic case, see Remark \ref{models-example1}. The second result, Theorem \ref{thm-non-modelli}, is a generalization of Theorem \ref{thm-modelli}, in that it allows the manifold to be non radially symmetric, provided it is not too far from being so in terms of curvature. It is important to notice that the case of a negative curvature with quadratic growth is included in this theorem as well, see Corollary \ref{thmuniq-qq}. Our last result, Theorem \ref{thm-non-modelli-fabio}, ensures uniqueness on general Cartan-Hadamard manifolds under a slightly different assumption on the admissible growth rate of the data (and solutions). Such a growth condition coincides with the one appearing in Theorem \ref{thmexi} when the kind of curvature bound assumed on $M$ is exactly of power type, as in \eqref{power}, see Remarks \ref{rem-fb} and \ref{models-example1}. However, this equivalence \emph{fails} in the critical case $\gamma=2$, as the admissible growth rate for uniqueness identified in Theorem \ref{thm-non-modelli-fabio} turns out to be slower.

The paper is organized as follows. Section \ref{Stat} provides the basic geometric and analytic background. Section \ref{existence} contains the proofs of the main results concerning existence and blow-up, that is Theorems \ref{thmexi} and \ref{opt-blow}. Section \ref{uniqueness} is devoted to the proofs of all our uniqueness results, namely Theorems \ref{thm-modelli}, \ref{thm-non-modelli}, \ref{thm-non-modelli-fabio} and Corollary \ref{thmuniq-qq}.

\section{Preliminary material and statements of the main results}\label{Stat}

This section is devoted to a collection of essential basic facts on Riemannian manifolds (especially of Cartan-Hadamard type), mostly regarding the Laplace-Beltrami operator, polar coordinates and Laplacian/volume comparison theorems. We will then introduce a suitable functional space, which is very useful to our purposes, and state the main results concerning \eqref{e64}.

\subsection{Notations for Cartan-Hadamard manifolds and models}\label{RG}
Given a Riemannian manifold $M$ of dimension $ N \in \mathbb{N}_{\ge 2} $, we let $\Delta$
denote the Laplace-Beltrami operator (sometimes written $ \Delta_M $ to avoid ambiguity), $\nabla$ the Riemannian
gradient and $d\mu$ the Riemannian volume measure on $M$. We will use the symbol $ d\mu_{N-1} $ to denote the related \emph{surface} or \emph{boundary area}, namely the $ (N-1) $-dimensional Hausdorff measure associated with $ d\mu $.

Throughout we will only consider \emph{Cartan-Hadamard} manifolds, that is complete and simply connected Riemannian manifolds with everywhere nonpositive sectional curvature (see \cite{GW} for an excellent reference). It is well known that on any such a manifold the \emph{cut locus} of every point $ o \in M $ is empty, so that in particular it is possible to define {\it polar coordinates} with respect to $o$. As a result, all $ x \in M\setminus \{o\}$ can uniquely be represented by the positive quantity $r \equiv  r(x) := d(x, o)$, namely the geodesic distance to $ o $, and an angle $\theta\in \mathbb S^{N-1}$. If we let $B_R$ denote the geodesic ball of radius $R>0$ centered at $o$ and $ S_R:=\partial B_R $, it turns out that
\begin{equation}\label{n51}
\mu_{N-1}(S_R)\,=\, \int_{\mathbb S^{N-1}}A(R, \theta) \, d\theta_{\mathbb{S}^{N-1}}\,, \qquad \mu(B_R) = \int_0^R \mu_{N-1}(S_r) \, dr \, ,
\end{equation}
where $ d\theta_{\mathbb{S}^{N-1}} $ is the canonical volume measure on the unit sphere, for an appropriate positive and smooth enough function $A(r,\theta)$ which is related to the metric tensor, see \cite[Section 3.1]{Grig}. Moreover, the Laplace-Beltrami operator (or simply the \emph{Laplacian}) in polar coordinates reads
\begin{equation}\label{e1}
\Delta = \frac{\partial^2}{\partial r^2} + \mathsf{m}(r, \theta) \, \frac{\partial}{\partial r} + \Delta_{S_{r}} \, ,
\end{equation}
where
\begin{equation}\label{def-m}
\mathsf{m}(r, \theta):=\frac{\partial }{\partial r} \log A(r,\theta)
\end{equation}
and $ \Delta_{S_{r}} $ is the Laplace-Beltrami operator on the submanifold $ S_{r} $. Thanks to \eqref{e1}, one can easily verify that $ \mathsf{m}(r,\theta) $ is in fact the Laplacian of the {distance function} $ x \mapsto r(x) $.

A very special subclass of Cartan-Hadamard manifolds (or more in general of manifolds that possess a pole) is formed by the so-called \emph{Riemannian models} or \emph{model manifolds}, that is spherically symmetric manifolds for which the metric can be written as follows:
\begin{equation*}\label{e2}
ds^2  = dr^2+\psi(r)^2 \, d\theta^2 \, ,
\end{equation*}
where $d\theta^2$ is the canonical metric on $\mathbb S^{N-1}$ and $ \psi $ is a smooth positive function on $ (0,+\infty) $ such that
$$
\psi(0)=0 \qquad \text{and} \qquad \psi'(0)=1 \, ,
$$
see \cite[Section 3.1]{Grig}. In this case, we will write $ M \equiv  M_\psi $ and say that $ \psi $ is the \emph{model function} associated with $ M_\psi $. Due to the particular warped structure of the metric, it is not difficult to show that $ A(r,\theta) = \psi(r)^{N-1} $ and that \eqref{e1} takes the simpler form
\begin{equation}\label{mm43}
\Delta = \frac{\partial^2}{\partial r^2} + (N-1) \, \frac{\psi'(r)}{\psi(r)} \, \frac{\partial}{\partial r} + \frac1{\psi(r)^2} \, \Delta_{\mathbb S^{N-1}} \, .
\end{equation}
Note that, thanks to polar coordinates, it is possible to introduce \emph{radial functions}, namely functions on $ M $ of the type $ x \mapsto f(r(x)) $ for a suitable real function $ f $ defined on $ [0,+\infty) $. When no ambiguity occurs, for the sake of readability, we will not explicitly distinguish $ f(r) $ from $ f(r(x)) $. If a radial function $f_1$ is naturally defined on a manifold $ M_1 $ (typically a model), we will say that it is \emph{radially transplanted} to another manifold $ M_2 $ when considering the radial function $f_2$ on $ M_2 $ induced by the same real function.

\subsection{Curvatures and Laplacian comparison}\label{lapl-comp}

Given a pole $ o \in M $ as in Section \ref{RG} and any $x\in M\setminus\{o\}$, we let $\mathrm{Ric}_o(x)$ denote the \emph{Ricci curvature} at $x$ in the radial direction identified by $\frac{\partial}{\partial r}$ (also known as radial Ricci curvature). If $\omega$ is a pair of tangent vectors from $T_x M$ having the form
$\left(\tfrac{\partial}{\partial r} , V \right)  $, where $V$ is an arbitrary unit vector orthogonal to $\frac{\partial}{\partial r}$, we let $\mathrm{K}_{\omega}(x)$ denote the \emph{sectional curvature} at $ x $ corresponding to the $2$-section (or plane) spanned by $\omega$ (also known as radial sectional curvature).

There is a number of well-known \emph{comparison} results for the Laplacian (and more in general for the Hessian) of the distance function under curvature bounds, which involve model manifolds, see e.g.~\cite[Section 2]{GW} or \cite[Sections 1.2.3--1.2.5]{MRS}. More precisely, if
\begin{equation}\label{e3a}
\mathrm{K}_{\omega}(x) \leq - \frac{\psi''(r)}{\psi(r)} \qquad \forall x \equiv (r,\theta) \in \mathbb{R}^+ \times \mathbb{S}^{N-1}
\end{equation}
for some model function $\psi$ and all $2$-section $ \omega $ as above, then
\begin{equation}\label{e3}
\mathsf{m}(r, \theta) \geq (N-1) \, \frac{\psi'(r)}{\psi(r)} \qquad \forall (r,\theta) \in \mathbb{R}^+ \times \mathbb{S}^{N-1} \, .
\end{equation}
Since a Cartan-Hadamard manifold has a nonpositive sectional curvature, we have that \eqref{e3a} is trivially satisfied with the choice
$\psi(r)=r$,  therefore \eqref{e3} entails
\begin{equation}\label{e5}
\mathsf{m}(r, \theta) \geq  \frac{N-1}{r} \qquad \forall (r,\theta) \in \mathbb{R}^+ \times \mathbb{S}^{N-1} \, .
\end{equation}
Conversely, if
\begin{equation}\label{e3c}
\mathrm{Ric}_{o}(x) \geq -(N-1) \, \frac{\psi''(r)}{\psi(r)} \qquad \forall x \equiv (r,\theta) \in \mathbb{R}^+ \times \mathbb{S}^{N-1}
\end{equation}
for some model function $\psi $, then
\begin{equation}\label{e4}
\mathsf{m}(r, \theta) \leq (N-1) \, \frac{\psi'(r)}{\psi(r)} \qquad \forall (r,\theta) \in \mathbb{R}^+ \times \mathbb{S}^{N-1} \, .
\end{equation}
Note that \eqref{e4} holds in greater generality, whereas for \eqref{e3} it is crucial that $ M $ is a Cartan-Hadamard manifold or at least that the cut locus of $o$ is empty, see again \cite{MRS}. Under \eqref{e3a}, we can infer that
\begin{equation*}
\mu_{N-1}(S_R) \geq \left| \mathbb{S}^{N-1} \right| \psi(R)^{N-1} \qquad \forall R>0 \, ,
\end{equation*}
while under \eqref{e3c} it follows that
\begin{equation}\label{n50a}
\mu_{N-1}(S_R) \leq \left| \mathbb{S}^{N-1} \right| \psi(R)^{N-1} \qquad \forall R>0 \, ,
\end{equation}
where $\left| \mathbb{S}^{N-1} \right|  $ stands for the volume of the $ (N-1) $-dimensional unit sphere. Such inequalities can be deduced from \eqref{e3} and \eqref{e4}, respectively, recalling \eqref{n51} and \eqref{def-m}. Estimate \eqref{n50a} is in fact a special case of the celebrated \emph{Bishop-Gromov} comparison theorem \cite[Section 1.2.4]{MRS}.

We point out that, if  $M \equiv  M_\psi$ is a model manifold, then for any $ x\equiv(r, \theta) $ inequalities \eqref{e3a} and \eqref{e3c} actually become identities:
\begin{equation}\label{e1ce}
\mathrm{K}_{\omega}(x)=-\frac{\psi''(r)}{\psi(r)} \, , \qquad \mathrm{Ric}_{o}(x)=-(N-1) \, \frac{\psi''(r)}{\psi(r)} \, .
\end{equation}
Moreover, the sectional curvature w.r.t.~planes \emph{orthogonal} to $\frac{\partial}{\partial r}$ is equal to
\begin{equation}\label{mm36}
\frac{1-\left[\psi'(r)\right]^2}{\psi(r)^2} \, .
\end{equation}
In particular, from \eqref{e1ce}--\eqref{mm36} it is readily seen that a necessary and sufficient condition for a model function $ \psi $ to give rise to a Cartan-Hadamard manifold $ M_\psi $ is $ \psi'' \ge 0 $, or equivalently that $ \psi $ be convex.

A quantity that will have a central role in our analysis is the following:
\begin{equation}\label{H-definition}
H(r) := \int_0^r \frac{\int_0^s \psi(t)^{N-1} \, dt}{\psi(s)^{N-1}} \, ds \qquad \forall r \ge 0 \, ,
\end{equation}
namely the primitive of the ``volume/surface area ratio'' function of the model manifold $ M_\psi $ associated with $ \psi $. Note that $ H $ is by construction nondecreasing and, as observed in \cite[Lemma 4.2]{GMP-transac}, satisfies the key inequality
\begin{equation}\label{eq-der-H}
\left[ H'(r) \right]^2 \le 2 H(r) \qquad \forall r \ge 0
\end{equation}
provided $ \psi $ is nondecreasing, which is always the case if $ M_\psi $ is Cartan-Hadamard. Moreover, the condition
\begin{equation}\label{eq-H-unbounded}
\lim_{r \to +\infty} H(r) = +\infty
\end{equation}
turns out to be equivalent to the \emph{stochastic completeness} of $ M_\psi $ \cite[Proposition 3.2]{Grig}. We refer to \cite[Appendix B]{GIM} for a brief review of the connections between such an important property of a (general) manifold and curvature or volume bounds. Accordingly, we will say that $ M_\psi $ is an $ N $-dimensional stochastically complete model manifold if \eqref{eq-H-unbounded} holds.

\subsection{Growth conditions and functional spaces}\label{sf}

If $ M $ is an $N$-dimensional Cartan-Hadamard manifold that satisfies the curvature bound \eqref{e3c}, then, as we will see in the sequel, the ``optimal'' class of data for which  problem \eqref{e64} is well posed is given by all measurable functions $ f $ on $M$ such that
\begin{equation}\label{e40z}
\left|f(x)\right| \leq C \left[ H(r(x))+1\right]^{\frac 1 {m-1}} \qquad \text{for a.e. } x \in M \, ,
\end{equation}
for some $ C \ge 0 $ that in general depends on $f$, where $ H $ is defined in \eqref{H-definition}. We let $ X_{\infty,\psi} $ denote such a functional space; note that, by virtue of \eqref{eq-H-unbounded}, in the stochastically-complete case \eqref{e40z} becomes a constraint on the possible growth rate of $ | f(x) | $ as $ r(x)\to+\infty $.  For later purposes, for any value of the parameter $ b \ge 0 $ we endow $ X_{\infty,\psi} $ with the norm
\begin{equation*}\label{e40-norm}
\left\| f \right\|_{\infty,b} := \left\| \frac{f}{\left( H+1+b \right)^{\frac 1 {m-1}}} \right\|_{L^\infty(M)} ,
\end{equation*}
which makes it a Banach space. It is readily seen that all such norms are equivalent, and by the definition of $ \| \cdot \|_{b,\infty} $ we have
\begin{equation*}\label{e40-bis}
\left|f(x)\right| \leq \left\| f \right\|_{\infty,b}  \left[ H(r(x))+1+b\right]^{\frac 1 {m-1}} \qquad \text{for a.e. } x \in M \, ,
\end{equation*}
so that $ \| f \|_{\infty,0} $ is the smallest constant $ C \ge 0 $ for which \eqref{e40z} holds. Unless otherwise specified, when referring to $ X_{\infty,\psi} $ as a Banach space we will implicitly assume that it is endowed with $ \| \cdot \|_{\infty,0} $. Finally, a calculus exercise shows that $ b \mapsto \| f \|_{\infty,b} $ is nonincreasing and the following identity holds:
\begin{equation}\label{e40-limsup-final}
\lim_{b \to +\infty}  \left\| f \right\|_{\infty,b} = \underset{r(x)\to+\infty}{\operatorname{ess}\limsup} \, \frac{\left| f(x) \right|}{\left[ H(r(x)) \right]^{\frac 1 {m-1}}} \, .
\end{equation}
In the sequel, in order to lighten the notation, at some points we will drop the explicit dependence of a function on its variables, at least when no ambiguity occurs.

\subsection{Existence, blow-up and uniqueness results}\label{sect: exuni}

Let us first provide our definition of (very weak) solution to \eqref{e64}. Throughout, when simply referring to ``solutions'', we will implicitly mean those satisfying \eqref{e64} according to the following.

\begin{den}\label{defsol}
Let $T>0$ and $ u_0 \in L^\infty_{{loc}}(M) $. We say that $u\in L^\infty_{{loc}}(M\times [0, T))$ is a very weak (or distributional) solution to problem \eqref{e64} if the identity
\begin{equation}\label{q50}
-\int_0^T \int_M u\, \varphi_t\,  d\mu dt  = \int_0^T \int_M u^m \, \Delta \varphi\, d\mu dt + \int_M u_0(x) \, \varphi(x,0)\, d\mu(x)
\end{equation}
holds for all $\varphi\in C^\infty_c(M\times [0, T))$.
\end{den}

When referring to ``subsolutions'' or ``supersolutions'' instead, we mean functions satisfying \eqref{q50} with the equality sign replaced by ``$ \le $'' and ``$ \ge $'', respectively, for arbitrary nonnegative test functions.

If $ u $ is a solution according to the above definition, we will say that it admits an \emph{extension} if there exists some $ T_\ast>T $ and a function $ v \in L^\infty_{{loc}}(M\times [0, T_\ast)) $ satisfying \eqref{q50} up to $ T = T_\ast $, such that $ u = v $ a.e.~in $ M \times (0,T) $.

We can now state our main existence result.

\begin{thm}[Existence]\label{thmexi}
Let $ M $ be an $ N $-dimensional Cartan-Hadamard manifold such that
\begin{equation}\label{thmexi-ricci}
\mathrm{Ric}_o(x) \ge -(N-1) \, \frac{\psi''(r)}{\psi(r)} \qquad \forall x \equiv (r,\theta) \in \mathbb{R}^+ \times \mathbb{S}^{N-1}
\end{equation}
for some $ o \in M $, where $\psi$ is the model function associated with an $N$-dimensional stochastically complete Cartan-Hadamard model manifold. Let $u_0 \in X_{\infty,\psi}$ and fix $ b \ge 0 $. Then there exists a solution $u$
to problem \eqref{e64} with
\begin{equation*}\label{n2}
T = \frac{c_1}{\left\| u_0 \right\|_{\infty,b}^{m-1}} \,  ,
\end{equation*}
%
where $c_1$ is a positive constant depending only on $ m $. In addition, the estimate
\begin{equation}\label{thm-norms}
\left\| u(t) \right\|_{\infty,b} \le c_2  \left( 1 - \frac{t}{T} \right)^{-\frac{1}{m-1}} \left\| u_0 \right\|_{\infty,b} \qquad \text{for a.e. } t \in (0,T)
\end{equation}
holds, where $ c_ 2>1 $ is another constant depending only on $m$.

Furthermore, the solution $u$ can be extended up to the time
\begin{equation}\label{thm-tmax}
{T}  = c_1 \left( \underset{r(x)\to+\infty}{\operatorname{ess}\limsup} \, \frac{\left| u_0(x) \right|}{ \left[H(r(x))\right] ^{\frac 1 {m-1}}} \right)^{-(m-1)}  ,
\end{equation}
such an extension being independent of $ b $. In particular, if
$$
\underset{r(x)\to+\infty}{\operatorname{ess}\limsup} \, \frac{\left| u_0(x) \right|}{ \left[H(r(x))\right] ^{\frac 1 {m-1}}}  = 0
$$
then problem \eqref{e64} admits a global solution.
\end{thm}

Estimates \eqref{thm-norms} and \eqref{thm-tmax} point to possible finite-time blow-up phenomena when the initial datum has a ``critical'' growth, in the sense that the $ \limsup $ in \eqref{thm-tmax} is positive. The next result shows that, in the class of model manifolds at hand, this is indeed the case at least for specifically selected initial data.

\begin{thm}[Blow-up]\label{opt-blow}
Let $ M_\psi $ be an $N$-dimensional stochastically complete Cartan-Hadamard model manifold, and let $ H $ be as in \eqref{H-definition}. Then, for every $ T>0 $ and $ \alpha>0 $, there exists a radial positive smooth solution $U_{T,\alpha}$ to
\begin{equation}\label{rad-1-state}
\Delta \!\left( U_{T,\alpha} \right)^m =  \frac{U_{T,\alpha}}{T(m-1)} \quad \text{in } M_\psi \, ,\qquad U_\alpha(0)=\alpha \, ,
\end{equation}
which is radially increasing and satisfies the two-sided bound
\begin{equation}\label{rad-2-state}
\frac{C_1}{T^{\frac 1{m-1}}} \left[ H(r)+1 \right]^{\frac{1}{m-1}} \le U_{T,\alpha}(r) \le \frac{C_2}{T^{\frac 1{m-1}}} \left[ H(r) + 1 \right]^{\frac{1}{m-1}} \qquad \forall r \ge 0
\end{equation}
for suitable positive constants $ C_1<C_2 $ depending only on $m$ and $ \alpha$. Therefore, the separable function
\begin{equation}\label{rad-3-state}
u_{T,\alpha}(x,t) := \left( 1 - \frac t T \right)^{-\frac{1}{m-1}} U_{T,\alpha}(r(x)) \qquad \forall (x,t) \in M_\psi \times [0,T)
\end{equation}
is a solution to \eqref{e64} with $ u_0 = U_{T,\alpha} $, belongs to $ L^\infty_{loc}\!\left([0,T);X_{\infty,\psi} \right) $ and blows up pointwise everywhere at $ t=T $.

Moreover, for every measurable initial datum $ u_0 $ such that $ U_{T,\alpha} \le u_0 \le U_{T,\beta} $ for some $ \beta>\alpha>0 $, one can construct a solution $u$ to \eqref{e64} satisfying $ u_{T,\alpha} \le u \le u_{T,\beta} $ in $ M_\psi \times [0,T) $, which, in particular, also belongs to $ L^\infty_{loc}\!\left([0,T);X_{\infty,\psi} \right) $ and blows up pointwise everywhere at $ t=T $.
\end{thm}

In order to prove that the solution constructed in Theorem \ref{thmexi} is unique, for our strategies to work condition \eqref{thmexi-ricci} is not enough, and we thus need to require some extra assumption. Accordingly, we present three different uniqueness results.

\begin{thm}[Uniqueness on models] \label{thm-modelli}
	Let $ M_\psi $ be an $N$-dimensional stochastically complete Cartan-Hadamard model manifold, and let $ H $ be as in \eqref{H-definition}. Assume in addition that there exist $ r_0,K>0 $ such that
	\begin{equation}\label{ipotesiH-enunciato}
	H'''(r) \le K \, \frac{H'(r)}{H(r)} \qquad \forall r \ge r_0 \, .
	\end{equation}	
	 Let $ u_0 \in X_{\infty,\psi} $. Then, for every $ T>0 $, there exists at most one solution $ u $ to \eqref{e64} such that  $ u \in L^\infty((0,T);X_{\infty,\psi}) $.
\end{thm}

The above result can be generalized to Cartan-Hadamard manifolds, complying with \eqref{thmexi-ricci}, that are \emph{not} required to be spherically symmetric (i.e.~models), provided they satisfy an additional curvature bound from below which allows for slight perturbations of models.

\begin{thm}[General uniqueness 1]\label{thm-non-modelli}
	Let $ M $ be an $ N $-dimensional Cartan-Hadamard manifold such that \eqref{thmexi-ricci} holds
	for some $ o \in M $, where $\psi$ is the model function associated with an $N$-dimensional stochastically complete Cartan-Hadamard model manifold, and let $ H $ be as in \eqref{H-definition}. Assume in addition that there exist $ r_0,K>0 $ such that \eqref{ipotesiH-enunciato} holds and
		\begin{equation}\label{sec}
	\textrm{K}_{\omega}(x)\leq -\frac{\phi''(r)}{\phi(r)} \qquad \forall x \equiv (r,\theta) \in \mathbb{R}^+ \times \mathbb{S}^{N-1} \, ,
	\end{equation}
	where $ \phi $ is another model function associated with a Cartan-Hadamard model manifold such that
		\begin{equation}\label{perturbation}
	\frac{\phi'(r)}{\phi(r)} \ge\frac{\psi'(r)}{\psi(r)}- K \, \frac{H'(r)}{H(r)} \qquad \forall r \ge r_0 \, .
	\end{equation}
	 Let $ u_0 \in X_{\infty,\psi} $. Then, for every $ T>0 $, there exists at most one solution $ u $ to \eqref{e64} such that  $ u \in L^\infty((0,T);X_{\infty,\psi}) $. 	
%
%
\end{thm}

In the critical case of a \emph{quadratic} negative curvature, as a consequence of Theorem \ref{thm-non-modelli} we obtain the following uniqueness result, which (partially) fills a gap left open in \cite[Theorem 2.7]{MP}, see also Remarks \ref{models-example1} and \ref{models-example-bis} below.

\begin{cor}\label{thmuniq-qq}
Let $ M $ be an $ N $-dimensional Cartan-Hadamard manifold such that
\begin{equation}\label{ricci-below-quadratic}
\mathrm{Ric}_o(x) \ge -(N-1)\, C_0 \left( 1+C_0 \, r^2  \right)  \qquad \forall x \equiv (r,\theta) \in  \mathbb{R}^+ \times \mathbb{S}^{N-1}
\end{equation}
and
\begin{equation}\label{ricci-above-quadratic-sect}
	\textrm{K}_{\omega}(x)\leq -C_0 \left( 1+C_0 \, r^2  \right) + \frac{K}{\log r}  \qquad \forall x \equiv (r,\theta) \in \left(r_0,+\infty \right)  \times \mathbb{S}^{N-1}
	\end{equation}
for some $ o \in M $ and $ r_0,C_0,K>0 $. Then, for every measurable initial datum $ u_0 $ such that
\begin{equation}\label{initial-log}
\left| u_0(x) \right| \le c \left[\log(r(x)+2)\right]^{\frac{1}{m-1}}  \qquad \text{for a.e. } x \in M
\end{equation}
for some $ c>0 $, there exists $ T>0 $ and a unique solution $ u $ to \eqref{e64} satisfying
\begin{equation}\label{initial-log-sol}
\left| u(x,t) \right| \le \tilde{c} \left[\log(r(x)+2)\right]^{\frac{1}{m-1}}   \qquad \text{for a.e. } (x,t) \in M \times (0,T)
\end{equation}
	for some $\tilde{c} \ge c$.
\end{cor}

Before stating our last uniqueness result, we need to introduce the following radial function, which in a sense should replace $ H(r) $ for large $r$:
\[
\mathcal B(r):=
\begin{cases}
\left[\frac{\psi(r)}{\psi'(r)}\right]^2 \log\!\left[\psi(r)\right] & \text{if } r\geq 2\,,\\
1 & \text{if }  r \in [0,2) \,.
\end{cases}\]

\begin{thm}[General uniqueness 2]\label{thm-non-modelli-fabio}
Let $ M $ be an $ N $-dimensional Cartan-Hadamard manifold such that \eqref{thmexi-ricci} holds for some $ o \in M $, where $\psi$ is the model function associated with an $N$-dimensional Cartan-Hadamard model manifold satisfying
\begin{equation}\label{log}
\log \psi(r) \le l \log \psi(r-1) \qquad \forall r \ge 3
\end{equation}
for some $ l>1 $.  Given $ T>0 $, let $u, v$ be any two solutions to problem \eqref{e64} corresponding to the same initial datum $ u_0 \in L^\infty_{loc}(M) $. Suppose in addition that $ u$ and $ v $ comply with the pointwise bound
\begin{equation}\label{e20z}
\left|u(x,t)\right| \vee \left|v(x,t)\right| \leq C \left[\mathcal B(r(x))\right]^{\frac{1}{m-1}} \qquad \text{for a.e. } (x,t) \in M \times (0, T)
\end{equation}
for some $ C>0 $, where $ a \vee b :=  \max\{ a,b \} $. Then $ u =  v$ a.e.~in $M \times (0, T)$.
\end{thm}

\begin{oss}\rm \label{rem-fb} The uniqueness class of solutions entailed by Theorem \ref{thm-non-modelli-fabio} should be compared with the existence one provided by Theorem \ref{thmexi}. We limit ourselves to observing that such classes coincide in the following significant example. Assume that, for some $C_0>0$ and $\sigma \in (0,2)$, it holds
\[
\operatorname{Ric}_o(x) \geq - C_0 \left( 1+ r(x) \right)^{2(1-\sigma)} \qquad \forall  x\in M\setminus\{o\} \, .
\]
As a consequence, it is readily seen that \eqref{thmexi-ricci} is fulfilled, for instance, by a suitable model function $\psi $ such that
\[
\psi(r)=e^{k r^{2-\sigma}}
\]
for some $k>0$ and $ r $ large enough, which clearly complies with \eqref{log}. Hence, a straightforward computation yields
\[\mathcal B (r)
= \tfrac{1}{k\left( 2-\sigma \right)^2} \, r^\sigma  \, ,
\]
still for $ r $ large enough. Therefore, Theorem \ref{thm-non-modelli-fabio} is in agreement with \cite[Theorem 2.3]{GMP-pures} ($ \gamma >-2 $ there). The fact that the optimal growth class, in this case, is $ \asymp r^{\sigma} $ can also be deduced by an alternative construction, see Remark \ref{models-example1}.
\end{oss}

\begin{oss}\rm
In fact it is also possible to show that, under suitable additional assumptions, if the initial datum $ u_0 $ is nonnegative and has a \emph{supercritical} growth at infinity, that is
$$
\underset{r(x)\to + \infty}{\operatorname{ess}\lim} \, \frac{u_0(x)}{\left[ H(r(x)) \right]^{\frac{1}{m-1}}} = + \infty \, ,
$$
then \emph{no nonnegative solution} to \eqref{e64} exists. The method of proof is essentially the same as in \cite[Theorem 2.5 and Corollary 2.6]{GMP-pures}, and relies on both the construction of a separable \emph{supersolution} of the form \eqref{rad-3-state} and \emph{uniqueness} results. To this end, Theorems \ref{thm-modelli}, \ref{thm-non-modelli} and Corollary \ref{thmuniq-qq} and will do, whereas Theorem \ref{thm-non-modelli-fabio} can be used provided the growth of $ \mathcal{B}(r) $ is compatible with $ H(r) $.
\end{oss}

\section{Existence and blow-up: proofs}\label{existence}

We start with a crucial result showing that there exist nontrivial radial solutions to a sublinear elliptic equation strictly related to \eqref{e64}, which satisfy an explicit two-sided estimate compatible with the class $ X_{\infty,\psi} $.

\begin{lem} \label{l1-rad}
Let $ m>1 $, and let $ M_\psi $ be an $N$-dimensional stochastically complete Cartan-Hadamard model manifold. Then there exists a radial positive smooth solution $U$ to
\begin{equation}\label{rad-1}
\Delta U =  U^{\frac 1 m} \qquad \text{in } M_\psi \, ,
\end{equation}
which is radially increasing and satisfies
\begin{equation}\label{rad-2}
C_1 \left[ H(r)+1 \right]^{\frac{m}{m-1}} \le U(r) \le C_2 \left[ H(r) + 1 \right]^{\frac{m}{m-1}} \qquad \forall r \ge 0
\end{equation}
for suitable positive constants $ C_1<C_2 $ depending only on $m$, where $ H $ is defined in \eqref{H-definition}.
\end{lem}
\begin{proof}
First of all, let us notice that, up to the normalization $ U(0)=1 $ (which is inessential), looking for a radial solution to \eqref{rad-1} amounts to solving the Cauchy problem
\begin{equation}\label{cauchy-1}
\begin{cases}
\frac{1}{\psi^{N-1}}\left( \psi^{N-1} \, U' \right)' = |U|^{\frac 1 m} \quad \text{in } (0+\infty) \, , \\
U(0)=1 \, , \\
U'(0)=0 \, .
\end{cases}
\end{equation}
The fact that there exists a global solution is straightforward, since the right-hand side is sublinear and $ \psi(r) $ is infinitesimally Euclidean as $ r \to 0 $. Moreover, multiplying the differential equation by $ \psi^{N-1} $ and integrating between $0$ and $r>0$ we obtain
\begin{equation}\label{eq-r3}
U'(r) = \frac{1}{\psi(r)^{N-1}} \, \int_0^r \psi(s)^{N-1} \left|U(s)\right|^{\frac 1 m} \, ds \qquad \forall r>0 \, ,
\end{equation}
so that the solution is increasing (thus it stays positive and we can remove moduli). This monotonicity property, in particular, entails
$$
U'(r) \le \frac{U(r)^{\frac 1 m}}{\psi(r)^{N-1}} \, \int_0^r \psi(s)^{N-1} \, ds \quad \Leftrightarrow \quad \left[ U(r)^{\frac{m-1}{m}} \right]' \le \frac{m-1}{m} \, H'(r) \qquad \forall r>0 \, ,
$$
whence, upon a further integration,
$$
U(r)^{\frac{m-1}{m}} \le 1 + \frac{m-1}{m} \, H(r) \qquad \forall r \ge 0 \, .
$$
From such estimate the upper bound in \eqref{rad-2} readily follows.

The proof of the lower bound is more involved, and it consists of a delicate generalization of an iterative argument exploited in \cite{GMP-pures}. To this aim let us observe that, since $ U(r) \ge U(0)=1 $, still from \eqref{eq-r3} it holds
\begin{equation}\label{low-est-zero}
U'(r) \ge \frac{1}{\psi(r)^{N-1}} \, \int_0^r \psi(s)^{N-1} \, ds \quad \Rightarrow \quad U(r) \ge H(r) \qquad \forall r > 0 \, .
\end{equation}
Consider now the general inequality
\begin{equation}\label{low-est-n}
U(r) \ge C_n \, H(r)^{p_n} \qquad \forall r \ge 0 \, ,
\end{equation}
where $ n \in \mathbb{N} $ and $  C_n,p_n $  are suitable positive real numbers (note that it is satisfied with $ n=0 $ and $ C_0=p_0=1 $ due to \eqref{low-est-zero}). Assuming that \eqref{low-est-n} holds, our goal is to show that also
\begin{equation}\label{low-est-n-plus}
U(r) \ge C_{n+1} \, H(r)^{p_{n+1}} \qquad \forall r \ge 0
\end{equation}
holds, upon choosing $ C_{n+1}>0$ and  $  p_{n+1} > p_n $ in an appropriate way. Indeed, the combination of \eqref{eq-r3} and \eqref{low-est-n} yields
\begin{equation}\label{eq-r4}
U'(r) \ge  \frac{C_n^{\frac 1 m}}{\psi(r)^{N-1}} \, \int_0^r \psi(s)^{N-1} H(s)^{\frac {p_n}{m}} \, ds \qquad \forall r>0 \, ;
\end{equation}
integrating by parts the on the right hand side, we obtain
$$
\int_0^r \psi(s)^{N-1} H(s)^{\frac {p_n}{m}} \, ds = H(r)^{\frac {p_n}{m}} \, \int_0^r \psi(s)^{N-1} \, ds - \frac {p_n}{m} \int_0^r H'(s) \, H(s)^{\frac{p_n}{m}-1} \int_0^s \psi(t)^{N-1} \, dt \, ds \, .
$$
Since
$$
H'(s) \, H(s)^{\frac{p_n}{m}-1} \int_0^s \psi(t)^{N-1} \, dt = \left[ H'(s) \right]^2 H(s)^{\frac{p_n}{m}-1} \, \psi(s)^{N-1} \le  2 \, H(s)^{\frac{p_n}{m}} \, \psi(s)^{N-1}
$$
(recall \eqref{eq-der-H}), it follows that
$$
\int_0^r \psi(s)^{N-1} H(s)^{\frac {p_n}{m}} \, ds \ge  H(r)^{\frac {p_n}{m}} \, \int_0^r \psi(s)^{N-1} \, ds - 2 \, \frac {p_n}{m} \int_0^r H(s)^{\frac{p_n}{m}} \, \psi(s)^{N-1} \, ds \, ,
$$
that is
\begin{equation}\label{eq-r5}
\int_0^r \psi(s)^{N-1} H(s)^{\frac {p_n}{m}} \, ds \ge \frac{m}{m+2p_n} \,  H(r)^{\frac {p_n}{m}} \, \int_0^r \psi(s)^{N-1} \, ds \ge \frac{m-1}{m+1} \,  H(r)^{\frac {p_n}{m}} \, \int_0^r \psi(s)^{N-1} \, ds \, ,
\end{equation}
where in the last passage we used the property $ p_n \le m/(m-1) $, which trivially follows from the previously proved upper bound. Going back to \eqref{eq-r4}, and taking advantage of \eqref{eq-r5}, we end up with
\begin{equation*}\label{eq-r6}
U'(r) \ge \frac{m-1}{m+1} \, \frac{C_n^{\frac 1 m}}{\psi(r)^{N-1}} \, H(r)^{\frac {p_n}{m}} \, \int_0^r \psi(s)^{N-1} \, ds =  \frac{m-1}{m+1} \, C_n^{\frac 1 m} \, H'(r) \, H(r)^{\frac {p_n}{m}} \qquad \forall r>0 \, .
\end{equation*}
By integrating this differential inequality, and exploiting again the fact that $ p_n \le m/(m-1) $, we find that
$$
U(r) \ge 1 + \frac{m-1}{m+1} \, \frac{1}{\frac{p_n}{m}+1}  \,  C_n^{\frac 1 m} \, H(r)^{\frac{p_n}{m}+1} \ge \underbrace{\frac{(m-1)^2}{m(m+1)}}_{=:\kappa_m} \, C_n^{\frac 1 m} \, H(r)^{\frac{p_n}{m}+1} \qquad \forall r \ge 0 \, .
$$
As a result, we deduce that \eqref{low-est-n} implies \eqref{low-est-n-plus} provided $ \{ C_n \} $ and $ \{ p_n \} $ satisfy the recurrence relations
$$
C_{n+1} = \kappa_m \, C_n^{\frac 1 m} \, , \qquad p_{n+1} = \frac{p_n}{m} + 1 \, ,
$$
which are equivalent to (recall that $ C_0=p_0=1 $)
$$
C_{n} = \kappa_m^{\frac{m-m^{-n+1}}{m-1}} \, , \qquad p_n=\frac{m}{m-1}-\frac{1}{(m-1) \, m^n} \, .
$$
Hence, with such choices, by induction we infer that \eqref{low-est-n} is true for every $ n \in \mathbb{N} $, namely
$$
U(r) \ge  \kappa_m^{\frac{m-m^{-n+1}}{m-1}} \, H(r)^{\frac{m}{m-1}-\frac{1}{(m-1) \, m^n} } \qquad \forall r \ge 0 \, ,
$$
and letting $ n \to \infty $ the claimed lower bound easily follows.
\end{proof}

Under suitable curvature assumptions, we can take advantage of the previous lemma to construct (elliptic and parabolic) barriers that will be very useful for the proof of both existence and uniqueness.

\begin{cor}\label{cor-a}
Let $m>1$, and let $ M $ be an $ N $-dimensional Cartan-Hadamard manifold such that
$$
\mathrm{Ric}_o(x) \ge -(N-1) \, \frac{\psi''(r)}{\psi(r)} \qquad \forall x \equiv (r,\theta) \in \mathbb{R}^+ \times \mathbb{S}^{N-1} \, ,
$$
for some $ o \in M $, where $\psi$ is the model function associated with an $N$-dimensional stochastically complete Cartan-Hadamard model manifold. Then, for every $ b \ge  0 $, the function
$$ \overline{U}:= U+b^{\frac{m}{m-1}}  $$
satisfies
\begin{equation}\label{cor-b}
\Delta \overline{U} \le  \overline{U}^{\frac 1 m} \qquad \text{in } M \, ,
\end{equation}
where $ U $ is the same function as in Lemma \ref{l1-rad}, radially transplanted (about $o$) to $ M $. Moreover, there exist two positive constants $K_1<K_2 $, depending only on $m$, such that
\begin{equation}\label{cor-c}
K_1 \left( H+1+b \right)^{\frac 1 {m-1}} \le \overline{U}^{\frac 1 m} \le K_2 \left( H+1+b \right)^{\frac 1 {m-1}} \qquad \text{in } M \, ,
\end{equation}
where $ H $ is defined in \eqref{H-definition}.
\end{cor}
\begin{proof}
By construction, the function $r \mapsto U(r)$ is increasing and satisfies \eqref{rad-1}. Therefore, thanks to Laplacian comparison (we refer in particular to \eqref{e4}), its radially-transplanted version on $M$, that we still denote by $U$, satisfies
$$
\Delta_M U = U'' + \mathsf{m}(r,\theta) \, U' \le U'' + (N-1) \frac{\psi'}{\psi} \, U' = \Delta_\psi U = U^{\frac 1 m} \qquad \text{in } M \, ,
$$
which entails
$$
\Delta_M \left( U+ b^{\frac{m}{m-1}} \right) = \Delta_M U \le  U^{\frac 1 m}  \le \left(  U +  b^{\frac{m}{m-1}} \right)^{\frac{1}{m}} \qquad \text{in } M \, ,
$$
that is \eqref{cor-b}. Finally, by combining \eqref{rad-2} and the numerical inequality
$$ 2^{-\frac 1 m} \left(A+B\right) \le  \left[ A^{\frac{m}{m-1}} + B^{\frac{m}{m-1}} \right]^{\frac{m-1}{m}} \le A+B \qquad \forall A,B \ge 0 \, , $$
we readily obtain the two-sided bound \eqref{cor-c}.
\end{proof}

\begin{lem}\label{lemm-A}
Let $ u_0 \in X_{\infty,\psi} $, and let the assumptions of Corollary \ref{cor-a} be met. Then, for every $b \ge 0$, the function
\begin{equation}\label{lemm-A-ubar}
\overline{u}(x,t) := \left( 1 - \frac{t}{T} \right)^{-\frac{1}{m-1}} \frac{\left\| u_0 \right\|_{\infty,b}}{K_1}  \left[U(r(x))+ b^{\frac{m}{m-1}} \right]^{\frac 1 m} \qquad \forall (x,t) \in M \times [0,T)
\end{equation}
is a supersolution to problem \eqref{e64} provided
\begin{equation}\label{ee33}
T = \frac{K_1^{m-1}}{(m-1) \left\| u_0 \right\|_{\infty,b}^{m-1}} \, .
\end{equation}
\end{lem}
\begin{proof}
First of all let us observe that, in view of the definition of the norm $ \| \cdot \|_{\infty,b} $ and \eqref{cor-c}, it holds
\begin{equation}\label{ee11}
u_0(x) \le \left| u_0(x) \right| \le \left\| u_0 \right\|_{\infty,b} \left[ H(r(x))+1+b \right]^{\frac 1 {m-1}} \le \frac{\left\| u_0 \right\|_{\infty,b}}{K_1} \left[U(r(x))+ b^{\frac{m}{m-1}} \right]^{\frac 1 m}  = \overline{u}(x,0) \, ,
\end{equation}
for all $ x \in M $. Moreover, thanks to \eqref{cor-b} and \eqref{ee33}, we have
\begin{equation}\label{ee22}
\begin{aligned}
\overline{u}_t = & \left( 1 - \frac{t}{T} \right)^{-\frac{m}{m-1}} \frac{\left\| u_0 \right\|_{\infty,b}}{(m-1)\,T\,K_1}  \left(U+ b^{\frac{m}{m-1}} \right)^{\frac 1 m} \\
\ge & \left( 1 - \frac{t}{T} \right)^{-\frac{m}{m-1}} \frac{\left\| u_0 \right\|_{\infty,b}^m}{K_1^m} \, \Delta \left(U+ b^{\frac{m}{m-1}} \right) = \Delta \overline{u}^m
\end{aligned}
\end{equation}
in $  M \times (0,T) $. Hence, \eqref{ee11} and \eqref{ee22} yield the claim.
\end{proof}

We are now in position to prove our main existence result.

\begin{proof}[Proof of Theorem \ref{thmexi}]
The strategy is similar to the one employed in \cite{BCP} and \cite{GMP-pures}, so we will stress the key points only. First of all, we approximate the initial datum by means of the following double truncation:
$$
u_{i,j,0} := \left( u_0^+ \wedge i \right) - \left( u_0^- \wedge j \right)  \in L^\infty(M) \qquad \forall i,j \in \mathbb{N} \, .
$$
Then we construct, for every such data, a family of bounded solutions  $ \{ u_{i,j} \} $ to \eqref{e64} (with $ u_0 \equiv u_{i,j,0} $ and $ T=+\infty $) satisfying
\begin{equation}\label{h-dir-pre}
u_{i+1,j} \ge u_{i,j} \ge u_{i,j+1} \qquad  \forall i,j \in \mathbb{N}  \, ,  \text{ a.e. in } M \times \mathbb{R}^+ \,  .
\end{equation}
This can be performed, for instance, upon solving for every $ n \in \mathbb{N} $ the Cauchy-Dirichlet problems
\begin{equation}\label{h-dir}
\begin{cases}
\partial_t \! \left( u_{n,i,j} \right) = \Delta \! \left(u^m_{n,i,j}\right) & \textrm{in } B_n\times \mathbb{R}^+ \,, \\
u_{n,i,j} = -\left\|  u_0^- \wedge j  \right\|_{L^\infty(M)} & \textrm{on } \partial B_n \times \mathbb{R}^+ \,, \\
u_{n,i,j} = u_{i,j,0}  & \textrm{on } B_n \times \{0\}\, ,
\end{cases}
\end{equation}
and observing that, by standard comparison principles, the following inequalities hold:
\begin{equation*}\label{h-dir-1}
\left\|  u_0^+ \wedge i  \right\|_{L^\infty(M)} \ge  u_{n+1,i,j} \ge u_{n,i,j} \ge -\left\|  u_0^- \wedge j  \right\|_{L^\infty(M)} \qquad  \forall n,i,j \in \mathbb{N}  \, ,  \text{ a.e. in } B_n \times \mathbb{R}^+
\end{equation*}
and
\begin{equation}\label{h-dir-2}
u_{n,i+1,j} \ge u_{n,i,j} \ge u_{n,i,j+1} \qquad  \forall n,i,j \in \mathbb{N}  \, ,  \text{ a.e. in } B_n \times \mathbb{R}^+ \, .
\end{equation}
As a result, the sequence $ \{ u_{n,i,j} \}_{n} $ (set to $ -\left\|  u_0^- \wedge j  \right\|_{L^\infty(M)} $ outside $ B_n $) is bounded and monotone increasing. It is therefore not difficult to check that its pointwise limit $ u_{i,j} $ is also globally bounded and satisfies the very weak formulation
\begin{equation}\label{q50.ij}
-\int_0^{+\infty} \int_M u_{i,j} \, \varphi_t \,  d\mu dt  =  \int_0^{+\infty} \int_M u_{i,j}^m \, \Delta \varphi\, d\mu dt + \int_M u_{i,j,0}(x) \, \varphi(x,0) \, d\mu(x)
\end{equation}
for all $\varphi\in C^\infty_c(M\times [0, +\infty))$. Moreover, by passing to the limit in \eqref{h-dir-2} as $ n \to \infty $, we deduce \eqref{h-dir-pre}.

Thanks to Lemma \ref{lemm-A}, we know that the function $ \overline{u} $ as in \eqref{lemm-A-ubar} is a supersolution to \eqref{e64}, hence it is a fortiori a supersolution to \eqref{h-dir} for every $ n \in \mathbb{N} $, so that again by comparison we can infer that
$$
u_{n,i,j} \le \overline{u} \qquad  \forall n,i,j \in \mathbb{N}  \, ,    \text{ a.e. in } B_n \times (0,T) \, ,
$$
which entails
\begin{equation}\label{q51.ij}
u_{i,j} \le \overline{u} \qquad  \forall i,j \in \mathbb{N}  \, ,    \text{ a.e. in } M \times (0,T) \, .
\end{equation}
In fact, it is possible to establish an analogous bound from below. To this end, let us observe that $ -\overline{u} $ is a \emph{subsolution} to  \eqref{h-dir} (note that $ -\overline{u}^m = (-\overline{u})^m $ and recall \eqref{ee11}). Since $ \overline{u}(x,t) $ is increasing w.r.t.~both $ r(x) $ and $t$, and  $ \lim_{r(x) \to + \infty} \overline{u}(x,0) = + \infty $, it is plain that there exists $ \tilde{n} \in \mathbb{N} $ (depending on $j$) such that
$$
-\left\|  u_0^- \wedge j  \right\|_{L^\infty(M)} \ge -\overline{u} \qquad \text{ a.e. in } B_n \times (0,T) \, , \ \forall n \ge \tilde{n} \, .
$$
Hence, comparison between $ u_{n,i,j} $ and $ -\overline{u} $ ensures that
$$
u_{n,i,j} \ge -\overline{u} \qquad  \text{ a.e. in } B_n \times (0,T) \, , \ \forall n \ge \tilde{n} \, , \ \forall i,j \in \mathbb{N}  \, ,
$$
which entails
\begin{equation}\label{q51.ij-bis}
u_{i,j} \ge - \overline{u} \qquad  \forall i,j \in \mathbb{N}  \, ,    \text{ a.e. in } M \times (0,T) \, .
\end{equation}
Finally, we let first $ i \to \infty  $ taking advantage of the leftmost inequality in \eqref{h-dir-pre}, which guarantees that the sequence $ \{ u_{i,j} \}_i $ is increasing, and of \eqref{q51.ij}, thanks to which we can deduce that the corresponding pointwise limit $ u_j $ is locally bounded in $ [0,T) $. Passing to the limit in \eqref{q50.ij} yields
\begin{equation*}\label{q52.ij}
-\int_0^{T} \int_M u_{j} \, \varphi_t \,  d\mu dt  =  \int_0^{T} \int_M u_{j}^m \, \Delta \varphi\, d\mu dt + \int_M u_{j,0}(x) \, \varphi(x,0) \, d\mu(x)
\end{equation*}
for all $\varphi\in C^\infty_c(M\times [0, T))$, where $ u_{j,0} := u_0^+ - \left( u_0^- \wedge j \right)  $. We then let $ j \to \infty $, upon noticing that $ \{ u_j \} $ is decreasing and fulfills $ u_j \ge - \overline{u} $ a.e.~in $ M \times (0,T) $: the pointwise limit $ u:=\lim_{j \to \infty} u_j $ plainly satisfies \eqref{q50} and is therefore the sought solution.

Estimate \eqref{thm-norms} is a consequence of \eqref{q51.ij} and  \eqref{q51.ij-bis}, which entail $ |u| \le |\overline{u}| $ a.e.~in $ M \times (0,T) $, and the very definition of $ \overline{u} $ recalling \eqref{ee33} and the rightmost inequality in \eqref{cor-c}:
$$
\begin{aligned}
\left| u(x,t) \right| \le & \left( 1 - \frac{t}{T} \right)^{-\frac{1}{m-1}} \frac{\left\| u_0 \right\|_{\infty,b}}{K_1}  \left[U(r(x))+ b^{\frac{m}{m-1}} \right]^{\frac 1 m} \\
\le  & \, \frac{K_2}{K_1} \left( 1 - \frac{t}{T} \right)^{-\frac{1}{m-1}} \left\| u_0 \right\|_{\infty,b} \left[ H(r(x))+1+b \right]^{\frac 1 {m-1}} \quad \text{for a.e. } (x,t) \in M \times (0,T) \, ,
\end{aligned}
$$
so that \eqref{thm-norms} just follows from the definition of the norm $ \| \cdot \|_{\infty,b} $.

The above procedure is clearly independent of $b$, in the sense that if one takes $ \tilde b > b  $ then the supersolution $\overline{u} $ (with $b$ replaced by $ \tilde{b} $ in \eqref{lemm-A-ubar}), and therefore the corresponding constructed solution $\tilde u$, is ensured to exist up to the (larger) time
$$
\tilde{T} = \frac{K_1^{m-1}}{(m-1) \left\| u_0 \right\|_{\infty,\tilde b}^{m-1}} \, ,
$$
but $ \tilde{u} $ agrees with $ u $ up to the time $ T $ associated with $b$. Hence, letting $ b \to +\infty $ (recalling \eqref{e40-limsup-final}), we end up with a well-defined solution up to the time associated with $ b=+\infty $, that is \eqref{thm-tmax}.
\end{proof}

Still the solution to the elliptic equation constructed in Lemma \ref{l1-rad}, up to routine transformations, gives rise to separable profiles that blow up everywhere in finite time.

\begin{proof}[Proof of Theorem \ref{opt-blow}]
The existence of a radial positive smooth solution $ U_{T,\alpha} $ to \eqref{rad-1-state} is a direct consequence of Lemma \ref{l1-rad}, as the choice $ U(0)=1 $ in the Cauchy problem \eqref{cauchy-1} was made for mere convenience. Replacing such an initial condition with
$$
U(0)= \left[(m-1) \, T \right]^{\frac{m}{m-1}}\alpha^m
$$
and setting
$$
U_{T,\alpha} = \left[(m-1) \, T \right]^{-\frac{1}{m-1}} U^{\frac 1 m}
$$
readily provides the claimed solution complying with \eqref{rad-2-state}. It is then an elementary computation to check that the function $ u_{T,\alpha} $ defined in \eqref{rad-3-state} solves \eqref{e64}, which clearly by construction belongs to $ L^\infty_{loc}([0,T);X_{\infty,\psi}) $ and blows up pointwise as $ t \to T^{-} $.

The last part of the statement follows from the fact that the solutions constructed in Lemma \ref{l1-rad} are ordered with respect to the initial condition: this is a consequence of pure ODE comparison results (see the proof of \cite[Lemma 5.5]{GMP-pures} for a rigorous argument). Hence, for initial data satisfying $ U_{T,\alpha} \le u_0 \le  U_{T,\beta} $, by reproducing the proof of Theorem \ref{thmexi} using $ u_{T,\alpha} $ and $  u_{T,\beta} $ as barriers in place of $ -\overline{u} $ and $ \overline{u} $, respectively, we readily obtain a solution $ u $ to \eqref{e64} such that $ u_{T,\alpha} \le u \le u_{T,\beta} $ in $ M_\psi \times [0,T) $, so that $ u $ is forced to blow up at the same time $ t=T $.
\end{proof}

\section{Uniqueness: proofs}\label{uniqueness}

Since the method of proof of Theorems \ref{thm-modelli} and \ref{thm-non-modelli} is essentially different from the one of Theorem \ref{thm-non-modelli-fabio}, we split this section into two independent subsections, in which we address the two strategies separately.

\subsection{First strategy}

Before proving Theorems \ref{thm-modelli} and \ref{thm-non-modelli}, we need a few preliminary technical results, which will be key to the whole strategy. The latter relies on providing a positive function $ z $ satisfying the \emph{supersolution} inequality
	\begin{equation*}\label{barrier-pre}
\Delta z \le \kappa \, \frac{z}{H+1} \qquad \text{in } M \, ,
\end{equation*}
for some $  \kappa>0 $, and enjoying suitable integrability properties. Such a function will then be tested in the very weak formulation satisfied by two \emph{ordered} solutions to the same problem \eqref{e64}, yielding an integral inequality that admits the constant $ 0 $ as its unique solution.

\begin{lem}\label{h-terza-h-seconda}
	Let $ \psi $ be the model function associated with an $ N $-dimensional Cartan-Hada\-mard model ma\-nifold, and let $ H $ be as in \eqref{H-definition}. Assume in addition that there exist $ r_0,K>0 $ such that
	\begin{equation}\label{ipotesiH-lemma}
	H'''(r) \le K \, \frac{H'(r)}{H(r)} \qquad \forall r \ge r_0 \, .
	\end{equation}	
	Then there exists $ \widehat{K}>0 $, depending only on $K$, such that
	\begin{equation}\label{ipotesiH-secondo}
	-H''(r) \le \widehat{K} \qquad \forall r \ge r_0 \, .
	\end{equation}	
	\end{lem}
\begin{proof}
	For every $ r \ge r_0 $ and $ s>r $, we have:
	$$
	\begin{aligned}
	H'(s) = H'(r) + \int_{r}^{s} H''(t) \, dt = &  \, H'(r) + \int_{r}^{s} \left[ H''(r) + \int_r^t H'''(\tau) \, d\tau  \right] dt \\
	\le & \, H'(r) + (s-r) H''(r) + {K} \int_r^s  \int_r^t  \frac{H'(\tau)}{H(\tau)}  \, d\tau \, dt \\
	= &  \, H'(r) + (s-r) H''(r) + {K} \int_r^s  \int_r^t  \left[ \log H(\tau)  \right]' d\tau \, dt \\
	= &  \, H'(r) + (s-r) H''(r) + {K} \int_r^s   \log\!\left( \frac{H(t)}{H(r)} \right) dt  \\
  	\le & \, H'(r) + (s-r) H''(r) + {K} \, (s-r)  \log\!\left( \frac{H(s)}{H(r)} \right) ,
	\end{aligned}
	$$
	where in the last passage we have exploited the monotonicity of $  H $. Since $ H'(s)>0 $, it follows that
	\begin{equation}\label{eq-H-second-last}
	-H''(r) < \frac{H'(r)}{s-r} + {K}  \log\!\left( \frac{H(s)}{H(r)} \right)  .
	\end{equation}
	On the other hand, by integrating inequality \eqref{eq-der-H} between $ r $ and $s$, we obtain:
	$$
  \sqrt{H(s)} - \sqrt{H(r)}	\le \frac{s-r}{\sqrt{2}} \, ;
	$$
	in particular, upon choosing $ s=r+\sqrt{H(r)} $, we infer that
	$$
	 \sqrt{H\!\left(r+\sqrt{H(r)}\right)} \le \frac{1+\sqrt 2}{\sqrt 2} \, \sqrt{H(r)}	\, .
	$$
	Hence, by plugging such an estimate in \eqref{eq-H-second-last} with $ s = r+\sqrt{H(r)}  $, and using again \eqref{eq-der-H}, we end up with
	$$
	-H''(r) < \frac{H'(r)}{\sqrt{H(r)}} + {K}  \log\!\left( \frac{H\!\left( r+\sqrt{H(r)} \right)}{H(r)} \right)  \le \sqrt 2 + 2K \log\! \left( \frac{1+\sqrt 2}{\sqrt 2} \right) ,
	$$
	that is \eqref{ipotesiH-secondo}.
	\end{proof}

\normalcolor

For a given model function $ \psi $ and a parameter $\alpha>0$ to be fixed later, let us now define
\begin{equation}\label{barrier}
w(r):=\frac{H'(r)}{H(r)^\alpha} \, \psi(r)^{-(N-1)} \qquad \forall r>0 \, ,
\end{equation}
where $H$ is as in \eqref{H-definition}. Such a function will take a central role in the construction of the above barrier $z$.

\begin{lem}\label{lemma-barriera}
	Let $ M_\psi $ be an $N$-dimensional stochastically complete Cartan-Hadamard model manifold, and let $ H $ be as in \eqref{H-definition}. Assume in addition that   \eqref{ipotesiH-lemma} holds for some $ r_0,K>0 $. Then, for every $ \alpha>0 $, there exists a positive smooth function $z$ such that
	\begin{equation}\label{barrier2}
\Delta z \le \kappa \, \frac{z}{H+1} \qquad \text{in } M_\psi  \, ,
	\end{equation}
	where $ \kappa>0 $ is a suitable constant depending only on $ \alpha , H(r_0) , K $. Moreover, $z$ coincides with the function $w$ given in \eqref{barrier} outside $B_{r_0}(o)$.	
\end{lem}
\begin{proof}
	We start our proof by computing explicitly the Laplacian of $w$. To this end, a straightforward computation yields
	\begin{equation}\label{accasecondo}
	H''(r)	= 1-(N-1) \, \frac{\psi'(r)}{\psi(r)} \, H'(r) \qquad \forall r>0 \, ,
	\end{equation}
	namely the identity $ \Delta H = 1 $.
	From here on, for the sake of readability, we will drop the explicit dependence on the variable $ r>0 $. By the definition of $w$, it follows that
	\begin{equation}\label{wprimo}
	\begin{aligned}
	w'=& \, H'' H^{-\alpha} \, \psi^{-(N-1)} -\alpha \, H^{-\alpha-1} \left(H'\right)^2\psi^{-(N-1)}-(N-1) \, \psi^{-N}  \psi'  H' H^{-\alpha}\\
	=& \, H^{-\alpha} \, \psi^{-(N-1)}\left[H''-\alpha \, \frac{\left(H'\right)^2}{H}-(N-1) \, \frac{\psi'}{\psi} \, H'\right]\\
	=& \, H^{-\alpha} \, \psi^{-(N-1)} \left[2H''-1- \alpha \, \frac{\left(H'\right)^2}{H}\right],
	\end{aligned}
	\end{equation}
	where in the last step we have used \eqref{accasecondo}. Besides, by taking the second derivative, we obtain:
	\[
	\begin{aligned}
	w''=&\left[-\alpha \, H^{-\alpha-1} H' \, \psi^{-(N-1)}-(N-1) \, \psi^{-N}\psi' H^{-\alpha}\right]\left[2H''-1- \alpha \, \frac{\left(H'\right)^2}{H}\right]\\
	& \, +H^{-\alpha}\,\psi^{-(N-1)}\left[ 2H'''-\alpha \, \frac{2H'H''H-\left(H'\right)^3}{H^2} \right]\\
	=& \, H^{-\alpha} \, \psi^{-(N-1)} \left[-\alpha \frac{H'}{H}-(N-1)\frac{\psi'}{\psi}\right]\left[2H''-1-\alpha \, \frac{\left(H'\right)^2}{H}\right]\\
	& \, +H^{-\alpha}\,\psi^{-(N-1)}\left[ 2H'''-\alpha \, \frac{2H'H''H-\left(H'\right)^3}{H^2} \right] . \\
	\end{aligned}
	\]
	We can now compute the Laplacian of $w$, recalling \eqref{mm43} and exploiting the above expressions for $w' , w''$. By elementary algebraic calculations, we have:
	\begin{equation}\label{wq-lap-w}
	\begin{aligned}
	\Delta w = & \, w''+(N-1) \, \frac{\psi'}{\psi} \, w'\\
	=& \, H^{-\alpha} \, \psi^{-(N-1)}\left[-\alpha \frac{H'}{H}-(N-1)\frac{\psi'}{\psi}\right]\left[2H''-1-\alpha \,\frac{\left(H'\right)^2}{H}\right]\\
	& \, +H^{-\alpha} \, \psi^{-(N-1)}\left[2H'''-2\alpha \, \frac{H'H''}{H}+\alpha \, \frac{\left(H'\right)^3}{H^2}\right]\\
	& \, +(N-1) \, \frac{\psi'}{\psi} \, H^{-\alpha} \, \psi^{-(N-1)}\left[2H''-1-\alpha \, \frac{\left(H'\right)^2}{H}\right]\\
	=& \, H^{-\alpha} \, \psi^{-(N-1)}\left[-4\alpha \, \frac{H'H''}{H}+\alpha\frac{H'}{H}+\alpha(\alpha+1) \, \frac{\left(H'\right)^3}{H^2}+2H'''\right].
	\end{aligned}
	\end{equation}
	We now require that the inequality
	\begin{equation*}\label{large r}
	\Delta w\le C \, \frac wH \qquad \text{in } M_{\psi} \setminus B_{r_0}(o)
	\end{equation*}
	be satisfied for a suitable constant $ C >0$ sufficiently large. In view of \eqref{wq-lap-w}, this amounts to
	\[
	H^{-\alpha} \, \psi^{-(N-1)}\left[-4\alpha \, \frac{H'H''}{H}+\alpha\frac{H'}{H}+\alpha(\alpha+1)\,\frac{\left(H'\right)^3}{H^2}+2H'''\right]\le C \, \frac wH=C \, H^{-\alpha} \, \psi^{-(N-1)}\frac{H'}{H} \, ,
	\]
	namely
	\begin{equation}\label{ineq}
	-4\alpha \, \frac{H'H''}{H}+\alpha\frac{H'}{H}+\alpha(\alpha+1) \, \frac{\left(H'\right)^3}{H^2}+2H'''\le C \, \frac{H'}{H} \, ,
	\end{equation}
	for all $ r \ge r_0 $. Thanks to \eqref{eq-der-H} and Lemma \ref{h-terza-h-seconda}, it is readily seen that \eqref{ineq} holds provided
	$$
	C \ge 2K + 2\alpha(\alpha+1) + \alpha+4\alpha\widehat{K} \, .
	$$
	In order to extend the barrier inside $ B_{r_0}(o) $, it is useful to observe that $ w $ is decreasing. Indeed, for all $ r>0 $, an integration by parts yields (recall that $ \psi'' \ge 0 $)
		\begin{equation*}\label{ineq-ee}
	\begin{aligned}
	(N-1)\frac{\psi'(r)}{\psi(r)}H'(r) = & \,  (N-1) \frac{\psi'(r)\int_0^r \psi(s)^{N-1} \, ds }{\psi(r)^{N}} \\
	 = & \,  \frac{N-1}{N} + \frac{(N-1) \, \psi'(r)\int_0^r \frac{\psi(s)^{N} \psi''(s) }{\left[\psi'(s)\right]^2} \, ds}{N \, \psi(r)^{N}} \ge \frac 1 2 \, ,
	\end{aligned}
	\end{equation*}
	which, in view of \eqref{accasecondo} and \eqref{wprimo}, entails  the claimed monotonicity property. Hence, because $ w'(r_0)<0 $, we can finally observe that the function
	$$
	z(x):= \begin{cases}
	w(r(x)) & \text{in }  M_{\psi} \setminus B_{r_0}(o) \\
	w(r_0) & \text{in } B_{r_0} (o)
	\end{cases}
	$$
   satisfies (weakly) \eqref{barrier2} with $ \kappa = C +C/H(r_0) $. In order to make it a smooth supersolution, it is enough to apply a standard (radial) regularization procedure in a neighborhood of $ r=r_0 $ and replace $ \kappa $ with, for instance, $ 2 \kappa $.
\end{proof}

We now show that, under the same assumptions as in Theorem \ref{thmexi}, if there exist two solutions to \eqref{e64} then it is possible to construct a third one which is smaller than both.
\begin{lem}\label{lemma-ordering}
	Let $ M $ be an $N$-dimensional Cartan-Hadamard manifold such that
	$$
	\mathrm{Ric}_o(x) \ge -(N-1) \, \frac{\psi''(r)}{\psi(r)} \qquad \forall x \equiv (r,\theta) \in \mathbb{R}^+ \times \mathbb{S}^{N-1} \, ,
	$$
	for some $ o \in M $, where $\psi$ is the model function associated with an $N$-dimensional stochastically complete Cartan-Hadamard model manifold. Let $ u_0 \in X_{\infty,\psi} $ and $T>0$. Suppose that $ u $ and $ v $ are two solutions to  \eqref{e64} such that
	\begin{equation}\label{lemma-ordering-1}
	u,v \in L^\infty\left( (0,T) ; X_{\infty,\psi} \right) .
	\end{equation}
	Then there exists $ S \in (0,T) $, depending only on $m$ and
	\begin{equation*}\label{lemma-ordering-2}
	L:=\max \left\{  \| u \|_{L^\infty\left( (0,T) ; X_{\infty,\psi} \right)} ,  \| v \|_{L^\infty\left( (0,T) ; X_{\infty,\psi} \right)} \right\} ,
	\end{equation*}
	and a solution $ \tilde{u} \in L^\infty\left( (0,S) ; X_{\infty,\psi} \right)  $ to \eqref{e64} satisfying
	$$
	\tilde{u} \le u \quad \text{and} \quad \tilde{u} \le v \qquad \text{a.e.~in } M \times (0,S) \, .
	$$
\end{lem}
\begin{proof}
	Thanks to Lemma \ref{lemm-A} with $ b=0 $ and $ \| u_0 \|_{\infty,0} $ replaced $ L $, we find that the function
	$$
	\overline{u}(x,t) := \left( 1 - \frac{t}{S_0} \right)^{-\frac{1}{m-1}} \frac{L}{K_1}  \left[U(r(x)) \right]^{\frac 1 m} \, , \quad S_0 := \frac{K_1^{m-1}}{(m-1) L^{m-1}} \, ,
	$$
	is a supersolution to \eqref{e64} and complies with
	\begin{equation}\label{sup-S}
	\overline{u} \ge \left| u \right| \quad \text{and} \quad \overline{u} \ge \left| v \right|  \qquad \text{a.e.~in } M \times (0,S) \, , \qquad \text{where } S:= \min \left\{ \tfrac{S_0}{2} , T \right\} .
	\end{equation}
	Indeed, as a straightforward consequence of Lemma \ref{l1-rad}, the latter is actually a solution to the differential equation in \eqref{e64}. Moreover, by the definitions of $ \| \cdot \|_{\infty,0} $ and $ L $, and by \eqref{cor-c}, we have:
	$$
	\begin{aligned}
	\left| u(x,t) \right| \le \left\| u(t) \right\|_{\infty,0} \left[ H(r(x))+1 \right]^{\frac 1 {m-1}} \le \frac{L}{K_1} \left[U(r(x)) \right]^{\frac 1 m} \le & \left( 1 - \frac{t}{S_0} \right)^{-\frac{1}{m-1}} \frac{L}{K_1}  \left[U(r(x)) \right]^{\frac 1 m} \\
	= & \, \overline{u}(x,t)
	\end{aligned}
	$$
	for a.e. $ (x,t) \in M \times (0,S)$, and clearly the same bound holds for $ v $. Since $ S \le  S_0 / 2 $, still by \eqref{cor-c} it is plain that $ \overline{u} \in L^\infty\left( (0,S) ; X_{\infty,\psi} \right) $.
	
	In order to construct the claimed solution $ \tilde{u} $, we can proceed similarly to the proof of Theorem \ref{thmexi} and solve first, for every $ n \in \mathbb{N} $, the following Cauchy-Dirichlet problems on balls:
	\begin{equation}\label{h-dir-22}
	\begin{cases}
	\partial_t \! \left( u_{n} \right) = \Delta \! \left(u^m_{n}\right) & \textrm{in } B_n \times (0,S) \,, \\
	u_{n} = - \overline{u} & \textrm{on } \partial B_n \times (0,S) \,, \\
	u_{n} = u_{0}  & \textrm{on } B_n \times \{0\}\,.
	\end{cases}
	\end{equation}
	Thanks to what we have just established, it is apparent that $ -\overline{u} $ is a subsolution to \eqref{h-dir-22} whereas both $ u $ and $v$ are supersolutions. Hence, by virtue of local comparison principles for very weak sub/supersolutions  and to \eqref{sup-S}, we can infer that
	\begin{equation}\label{h-dir-23}
	\overline{u} \ge u,v \ge u_{n+1} \ge u_{n} \ge -\overline{u} \qquad  \text{a.e.~in } B_n \times (0,S) \, , \ \forall n \in \mathbb{N} \, .
	\end{equation}
	Note that \emph{very weak} local comparison principles are quite delicate and need several technicalities to be justified rigorously; this was addressed in detail in \cite[Section 3.4]{GMP-transac} in the case $ m<1 $, but actually the case $m>1$ is simpler since solutions need not be lifted (because the function $ u \mapsto u^m $ is locally Lipschitz). 
	
	Using estimates \eqref{h-dir-23}, which in particular ensure the monotonicity of the sequence $ \{ u_n \} $ (extended to $ -\overline{u} $ outside $ B_n $), it is readily seen that the pointwise limit $ \tilde{u}:=\lim_{n \to \infty} u_n $ is a solution to \eqref{e64} belonging to $ L^\infty\left( (0,S) ; X_{\infty,\psi} \right) $ and  complying with \eqref{lemma-ordering-1}.
\end{proof}

\normalcolor

We are now in position to prove Theorems \ref{thm-modelli} and \ref{thm-non-modelli}, by means of a method first introduced in \cite{GMP-advances} for \emph{bounded} solutions. In order to successfully adapt such a strategy, the structure of the barrier constructed in Lemma \ref{lemma-barriera} will be crucial.

\begin{proof}[Proof of Theorem \ref{thm-modelli}]
	Let $ u $ and $ v $ be any two solutions to \eqref{e64} satisfying \eqref{lemma-ordering-1}. Thanks to Lemma \ref{lemma-ordering}, we know that there exists another solution $ \tilde{u} $ to \eqref{e64}, defined at least up to a time $ S \in (0,T) $ that depends only on $ m $ and the quantity $ L $ as in \eqref{lemma-ordering-1}, which is smaller than both $ u $ and $ v $ and belongs to $ L^\infty((0,S);X_{\infty,\psi}) $. Hence, if we are able to show that $ u=\tilde{u}=v $ up to $ t=S $, then by iterating the procedure finitely many times we infer that $ u=v $ up to $ t=T $. For this reason, in the sequel we can and will assume without loss of generality that $  u \ge v $ a.e.~in $ M_\psi \times (0,T) $.
	
	In the definition of very weak solution, namely identity \eqref{q50} for both $ u $ and $v$, the choice $\phi(x,s)=\chi_{[0,t]}(s) \, \phi_R(x) $ (up to a routine time regularization) yields
	\begin{equation}\label{unique1}
	\int_{M_\psi}\left[u(\cdot,t)-v(\cdot,t)\right]\phi_R \, d \mu=\int_0^t \int_{M_\psi} \left(u^m-v^m\right) \Delta\phi_R \, d\mu ds \qquad \text{for a.e. } t \in (0,T) \, ,
	\end{equation}
    where, for every $R>0$, we pick
    $$\phi_R(x) =  \eta_R(r(x)) \, z(x)  \, , \quad \eta_R(r) := \eta\!\left(\frac{r}R\right) , $$
     $z$ being the barrier function constructed in Lemma \ref{lemma-barriera} (associated with a suitable $ \alpha $ that will be chosen below) and $\eta $ a smooth nonincreasing cut-off function such that $ 0 \le \eta \le 1 $ for all $ r \ge 0 $, $ \eta \equiv 1 $ for $ r \in [0,1] $ and $ \eta \equiv 0 $ for $ r \ge 2 $. A standard computation, combined with \eqref{barrier2}, entails
     \[
	\Delta\phi_R=\eta_R \, \Delta z+z \, \Delta\eta_R+2 \eta_R' \, z' \le \kappa \, \frac {\eta_R\,z}{H+1} + z \, \Delta\eta_R + 2\eta_R' \, z' \qquad \text{in } M_\psi \, ,
	\]
	where the prime symbol stands for derivative w.r.t.~$r$ (recall that $z$ is radial). Therefore, by plugging the above pointwise inequality in the r.h.s.~of \eqref{unique1}, we obtain:
	\begin{equation}\label{unique2}
	\begin{aligned}
	\int_0^t \int_{M_\psi} \left(u^m-v^m\right) \Delta\phi_R \,  d \mu  ds \le  & \,  \kappa \, \int_0^t \int_{M_\psi} \left(u^m-v^m\right) \frac{\eta_R\,z}{H+1}\, d\mu ds\\ & \, +\int_0^t\int_{M_\psi} \left(u^m-v^m\right) z \, \Delta \eta_R\,  d\mu ds \\
	& \, +2\int_0^t\int_{M_\psi} \left(u^m-v^m\right)\eta_R' \, z' \, d\mu ds \\
	\le &  \, \kappa \, \int_0^t \int_{M_\psi} \frac{u^m-v^m}{H+1} \, z \, d\mu ds \\
	     & \,  +\int_0^t\int_{M_\psi} \left(u^m-v^m\right) \left( z \, \eta''_R + 2 \eta_R' \, z' \right)  d\mu ds \, ,
	\end{aligned}
	\end{equation}
	where in the last inequality we used the fact that $ \eta_R \le 1$ and $ \Delta \eta_R \le \eta_R'' $, consequence of $ \eta_R' $ and \eqref{mm43}. Now we observe that, thanks to \eqref{lemma-ordering-1} and Lagrange's theorem, it holds (for some function $ \xi $ between $ v $ and $u$)
	\begin{equation}\label{unique2-bis}
	u^m-v^m=m\, \xi^{m-1} \left(u-v\right) \le C \left(H+1\right) \left(u-v\right) \le C \left(H+1\right)^{\frac{m}{m-1}} \qquad \text{a.e.~in } M_{\psi} \times (0,T) \, ,
	\end{equation}
	where, here and below, $C$ is a large enough positive constant independent of $ R $ and $t$ (whose explicit value is allowed to change from line to line). On the other hand, by the definition of $ \eta_R $ we readily have
		\begin{equation*}\label{unique2-ter}
		\left| \eta_R' \right|  \le \frac{C}{R} \, \chi_{[R,2R]} \, , \qquad \left| \eta_R'' \right| \le \frac{C}{R^2} \, \chi_{[R,2R]} \, .
		\end{equation*}
	Recalling that for $ r \ge r_0 $ the barrier $ z $ coincides with the function $ w $ defined in \eqref{barrier}, and exploiting \eqref{wprimo}, we thus deduce that
	$$
	\begin{aligned}
	\left| z \, \eta''_R + 2 \eta_R' \, z' \right| = & \, H^{-\alpha} \, \psi^{-(N-1)} \left|  H' \, \eta''_R + 2 \eta_R' \left[2H''-1- \alpha \, \frac{\left(H'\right)^2}{H}\right] \right| \\
	\le & \, C \, H^{-\alpha} \, \psi^{-(N-1)}  \left( \frac{H'}{R^2} + \frac 1 R  \right) \chi_{[R.2R]}
	\end{aligned}
	$$
	provided $ R $ is large enough, where in the last passage we have taken advantage of \eqref{eq-der-H} and the fact that $|H''|$ is bounded, consequence of  \eqref{accasecondo} and Lemma \ref{h-terza-h-seconda}. Since, on a model manifold, we have $ d\mu = \psi^{N-1} dr \otimes d\mu_{\mathbb{S}^{N-1}} $, the above inequality and \eqref{unique2-bis} yield (still for large $ R $)
	\[
	\begin{aligned}
	\left| \int_0^t\int_{M_\psi} \left(u^m-v^m\right) \left( z \, \eta''_R + 2 \eta_R' \, z' \right)  d\mu ds \right| \le & \, \frac C{R^2}\int_R^{2R}
	\frac{H'}{H^{\alpha-\frac m{m-1}}}\,{dr} + \frac C{R}\int_R^{2R}	\frac{1}{H^{\alpha-\frac m{m-1}}}\,{dr} \, .\\
	\end{aligned}
\]
If we now fix any $ \alpha>\frac{m}{m-1}+1 $ and use the fact that $ H(r) $ is increasing, we can bound from above the r.h.s~by the quantity
$$
\frac{C}{R^2 H(R)^{\alpha -\frac{m}{m-1} -1}} + \frac{C}{H(R)^{\alpha-\frac{m}{m-1}}} \, ,
$$
	which vanishes as $ R \to + \infty $ because $ M_\psi $ is stochastically complete, which is equivalent to $ \lim_{r\to+\infty}H(r)=+\infty $. As a result, going back to \eqref{unique2} and exploiting the middle inequality in \eqref{unique2-bis}, we infer that
	$$
 \limsup_{R \to + \infty}	\int_0^t \int_{M_\psi} \left(u^m-v^m\right) \Delta\phi_R \,  d \mu  ds \le C \, \int_0^t \int_{M_\psi} \left( u-v \right) z \, d\mu ds \, .
	$$
	It is readily seen that, with the above choices, $ (u-v)z \in L^1(M_\psi \times (0,T)) $. Therefore, passing to the limit as $R\to+\infty$ in \eqref{unique1} we end up with the integral inequality
	\begin{equation}\label{almost-last-int}
		\int_{M_\psi}\left[u(\cdot,t)-v(\cdot,t)\right] z \, d \mu \le C \, \int_0^t \int_{M_\psi} \left( u-v \right) z \, d\mu ds \qquad \text{for a.e. } t \in (0,T) \, ,
		\end{equation}
		which, upon integration with an exponential time factor, entails
		$$
		\int_0^t \int_{M_\psi} \left( u-v \right) z \, d\mu ds \le 0 \qquad \forall t \in (0,T) \, ,
		$$
		namely the thesis.
\end{proof}

\begin{oss}\label{models-example1} \rm A straightforward computation shows that, in fact,
	\[
	\psi(r) = C \, e^{\frac{1}{N-1} \, \int_1^r \frac{1-H''(s)}{H'(s)} \, ds }  \qquad \forall r \ge 1
	\]
	for a suitable $C>0 $. In particular, if we impose that for some $ \sigma \in [0,2) $ it holds
	\begin{equation}\label{ee-adv-1}
	H(r) \asymp \begin{cases} r^\sigma & \text{if } \sigma \neq 0 \, , \\ \log r & \text{if } \sigma=0 \, , \end{cases}  \ \ H'(r) \asymp r^{\sigma-1} \, , \ \ H''(r) = o(1)  \, , \ \  H'''(r) =\begin{cases} o\!\left( \frac{1}{r} \right) & \text{if } \sigma \neq 0 \, , \\  o\!\left( \frac{1}{r\log r} \right) & \text{if } \sigma=0 \, , \end{cases}
	\end{equation}
as $ r \to +\infty $, then $ \lim_{r \to +\infty} H(r) = +\infty $ and assumption \eqref{ipotesiH-enunciato} is satisfied. Moreover, another elementary calculation taking advantage of \eqref{ee-adv-1} yields
\begin{equation*}\label{ricci-general}
\frac{\psi''(r)}{\psi(r)} \asymp r^{2(1-\sigma)} \qquad \text{as } r \to + \infty \, ,
\end{equation*}
so that in particular, having in mind \eqref{thmexi-ricci} in Theorem \ref{thmexi}, on model manifolds we are able to treat, at the level of both existence and uniqueness, curvature bounds of the form (for any $ C_0>0 $)
\begin{equation}\label{ricci-general-two}
\operatorname{Ric}_o(x)\geq - C_0 \left( 1+ r(x) \right)^{2(1-\sigma)} \qquad \forall  x\in M\setminus\{o\} \, ,
\end{equation}
which were already addressed in \cite{GMP-pures}. However, such a paper did not cover the critical case $ \sigma=0 $, corresponding to a \emph{quadratic} negative curvature. The latter was actually covered in the subsequent paper \cite{MP}, but uniqueness was established under a stronger curvature bound from below involving a logarithmic correction of the squared distance (see in particular Theorem 2.7 there).
%
\end{oss}

\begin{proof}[Proof of Theorem \ref{thm-non-modelli}]
	We adapt the main ideas behind the proof of Theorem \ref{thm-modelli}. Since the overall strategy is analogous, we will only stress the key points. First of all, we observe that the function $w$ defined in \eqref{barrier} still satisfies the supersolution inequality
	\begin{equation}\label{supsol-phi}
		\Delta_\phi w\le C \, \frac wH \qquad \text{in } M_{\phi} \setminus B_{r_0} \, ,
	\end{equation}
	for a suitable constant $ C>0 $. Indeed, in radial coordinates \eqref{supsol-phi} reads
			\begin{equation}\label{technical3}
	w''(r)+(N-1) \, \frac{\phi'(r)}{\phi(r)} \, w'(r) \le C \, \frac{w(r)}{H(r)} \qquad \forall r \ge r_0 \, ,
	\end{equation}
	which is implied by
		\begin{equation*}\label{technical3-bis}
	w''(r)+(N-1)\left[\frac{\psi'(r)}{\psi(r)}-K \, \frac{H'(r)}{H(r)}\right]w'(r) \le C \, \frac{w(r)}{H(r)}  \qquad \forall r \ge r_0
	\end{equation*}
	thanks to \eqref{perturbation} and the fact that $ w $ is radially decreasing by construction. Hence, by means computations similar to the ones carried out in the proof of Lemma \ref{lemma-barriera} (recall \eqref{wprimo}), it is not difficult to check that \eqref{technical3} does hold provided
	$$
	C \ge  K\left( 2\alpha+3+2\widehat{K} \right) + 2\alpha(\alpha+1) + \alpha+4\alpha\widehat{K} \, .
	$$
	Finally, by arguing exactly as in the last part of the proof of Lemma \ref{lemma-barriera}, for a suitable $ \kappa>0 $ we can construct a global, positive smooth function $z$ satisfying
	\begin{equation}\label{barrier2-mody}
\Delta_\phi z \le \kappa \, \frac{z}{H+1} \qquad \text{in } M_\phi  \, ,
\end{equation}
	which is moreover radially nonincreasing and coincides with $ w $ outside $ B_{r_0}(o) $. By virtue of \eqref{sec} and Laplacian comparison (see in particular \eqref{e3}), we know that
	$$
  \mathsf{m}(r,\theta) \ge (N-1) \, \frac{\phi'(r)}{\phi(r)} \qquad	\forall (r,\theta) \in \mathbb{R}^+ \times \mathbb{S}^{N-1}  \, ,
	$$
	where $ \mathsf{m}(r,\theta) $ is the Laplacian of the distance function on $ M $, in radial coordinates. Hence, in view of the just mentioned monotonicity property of $z$ and formula \eqref{e1}, we readily deduce that the latter, radially transplanted from $ M_\phi $ to $ M $, satisfies the analogue of \eqref{barrier2-mody} on $ M $, namely
		\begin{equation*}\label{barrier2-mody-2}
	\Delta_M z \le \kappa \, \frac{z}{H+1} \qquad \text{in } M \, .
	\end{equation*}
	In order to rigorously complete the proof and obtain the integral inequality (for two ordered solutions $ u \ge v $)
	\begin{equation}\label{supersol-M-uv}
			\int_{M}\left[u(\cdot,t)-v(\cdot,t)\right] z \, d \mu \le \widehat{C} \, \int_0^t \int_{M} \left( u-v \right) z \, d\mu ds \qquad \text{for a.e. } t \in (0,T) \, ,
		\end{equation}
		for a suitable  $ \widehat{C}>0 $, we can exploit the same cut-off arguments as in the proof of Theorem \ref{thm-modelli}. To this aim, it is crucial to take advantage of the surface-comparison inequality \eqref{n50a}, consequence of \eqref{thmexi-ricci}. In particular, for any radial measurable function $f$ on $ M $ we have
		\begin{equation}\label{haus-ee}
 	 \int_M \left| f \right| d\mu	\le \left| \mathbb{S}^{N-1} \right| \int_{\mathbb{R}^+} \left| f(r) \right| \psi(r)^{N-1} \, dr  \, .
		\end{equation}
		Since all the functions involved in the computations from \eqref{unique1} to \eqref{almost-last-int} are radial or can be bounded pointwise by radial functions, using \eqref{haus-ee} it is straightforward to check that cut-off remainder integral terms still vanish as $ R \to +\infty $ (note that the inequality $ \Delta_M \eta_R \le \eta_R'' $ continues to hold since $ \mathsf{m}>0 $). Therefore, we end up with \eqref{supersol-M-uv}, from which the conclusion follows again by integration.
\end{proof}

Finally, in order to prove Corollary \ref{thmuniq-qq}, we need to exhibit some appropriate choices of functions $ \psi $ and $ \phi $ such that \eqref{ricci-below-quadratic} and \eqref{ricci-above-quadratic-sect} can be rewritten in the form \eqref{thmexi-ricci} and \eqref{sec}--\eqref{perturbation}, respectively.

\begin{proof}[Proof of Corollary \ref{thmuniq-qq}]
	First of all, we observe that the function
	$$
	\psi(r) = e^{\frac{C_0}{2} \, r^2 } \qquad \forall r \ge  0
	$$
	satisfies
	$$
	\frac{\psi''(r)}{\psi(r)} = C_0 \left( 1 + C_0 \, r^2 \right) \qquad \forall r > 0 \, ,
	$$
	however it is not an admissible model function because it does not comply with $ \psi(0)=0 $ and $ \psi'(0)=1 $. Nevertheless, it is not difficult to check that one can choose three constants $ a,b,c>0 $ in such a way that the function (which we still call $ \psi $ for notational convenience)
		\begin{equation*}\label{psi-second-hyp}
	{\psi}(r) =
	\begin{cases}
	 \frac 1 c \sinh(c r)  & \text{for } r \in [0,1) \, , \\
	   a \, e^{\frac{C_0}{2} \, r^2 } - b & \text{for } r \ge 1 \, ,
	\end{cases}
	\end{equation*}
	is $ C^1([0,+\infty)) \cap C^2([0,1) \cup (1,+\infty))  $, satisfies $ \psi(0)=0 $, $ \psi'(0)=1 $, and
	\begin{equation*}\label{psi-second-ex}
	\frac{\psi''(r)}{\psi(r)}  \ge  C_0 \left( 1 + C_0 \, r^2 \right) \qquad \forall r \in (0,1) \cup (1,+\infty) \, .
	\end{equation*}
	In particular, due to \eqref{ricci-below-quadratic}, we can infer that \eqref{thmexi-ricci} holds. The fact that $ \psi''(r) $ has a jump at $ r=1 $ is irrelevant to our purposes (it can be overcome by a standard local regularization argument), hence we will disregard it from now on. Moreover, from the explicit expression of $ \psi $ we obtain
		\begin{equation}\label{psi-second-prime}
	\frac{\psi'(r)}{\psi(r)}  = C_0 \, r + \frac{b \, C_0 \, r}{ a \, e^{\frac{C_0}{2} \, r^2 } - b} \qquad \forall r \ge 1 \, .
	\end{equation}
	Furthermore, an elementary but lengthy computation, in the spirit of Remark \ref{models-example1}, shows that the function $ H $ associated with $ \psi  $ is such that $ H(r) \asymp \log r $ and $ H'(r)  \asymp \frac 1 r  $ as $ r \to +\infty $. As a consequence, for an initial datum $ u_0 $ as in \eqref{initial-log}, we can apply Theorem \ref{thmexi} to assert that there exists at least one solution to \eqref{e64} fulfilling \eqref{initial-log-sol}. In order to deal with uniqueness, we take advantage of Theorem  \ref{thm-non-modelli}. To this purpose, consider the function
	$$
	\phi(r) = A \, \frac{e^{\frac{C_0}{2} \, r^2 }}{\left( \log r \right)^\kappa} + B  \, ,
	$$
	defined for every $ r >1 $, where $ A,B>0$ are suitable constants to be chosen and $ \kappa > K/C_0 $ is fixed. Direct calculations yield
			\begin{equation}\label{psi-second-prime-bis}
			 \frac{\phi'(r)}{\phi(r)} \ge C_0 \, r - \frac{B \, C_0}{A} \frac{r \left( \log r \right)^\kappa}{e^{\frac{C_0}{2} \, r^2 }}   -  \frac{\kappa}{r \log r}
			\end{equation}
	and
				\begin{equation}\label{psi-second-prime-ter}
				\begin{aligned}
\frac{\phi''(r)}{\phi(r)} = \frac{A \left[ \frac{e^{\frac{C_0}{2} \, r^2 }}{\left( \log r \right)^\kappa}  \right]''}{A \, \frac{e^{\frac{C_0}{2} \, r^2 }}{\left( \log r \right)^\kappa} + B } \le \frac{ \left[ \frac{e^{\frac{C_0}{2} \, r^2 }}{\left( \log r \right)^\kappa}  \right]''}{ \frac{e^{\frac{C_0}{2} \, r^2 }}{\left( \log r \right)^\kappa}  } = & \, C_0 \left( 1 + C_0 \, r^2 \right) + \frac{\kappa}{r^2  \log r  } \left( 1 + \frac{\kappa+1}{\log r} \right) - \frac{2\kappa \, C_0}{\log r} \, ,  \\
\le & \, C_0 \left( 1 + C_0 \, r^2 \right) - \frac{K}{\log r}
\end{aligned}
	\end{equation}
	provided $ r>R $ for a suitable $ R>1 \vee r_0 $ large enough independent of $ A $ and $B$. Although such a $ \phi $ is not an admissible model function near the origin, it is possible to find $ R_0 > R $  and $ A,B>0 $ in such a way that the new function (that we do not relabel)
	 		\begin{equation*}\label{psi-second-hyp-bis}
	 {\phi}(r) =
	 \begin{cases}
	 r  & \text{for } r \in [0,R_0) \, , \\
	 A \, \frac{e^{\frac{C_0}{2} \, r^2 }}{\left( \log r \right)^\kappa} + B & \text{for } r \ge R_0 \, ,
	 \end{cases}
	 \end{equation*}
	is admissible and $ C^1([0,+\infty)) \cap C^2([0,R_0) \cup (R_0,+\infty))  $, which in particular satisfies \eqref{psi-second-prime-ter} for $ r>R_0 $ and $ \phi''(r)=0 $ for $ r \in [0,R_0) $ (again, the fact that the second derivative has a jump at $ r=R_0 $ is inessential). Due to \eqref{ricci-above-quadratic-sect}, we therefore deduce that \eqref{sec} holds with this choice. Finally, thanks to \eqref{psi-second-prime}, \eqref{psi-second-prime-bis} and the just mentioned asymptotic behavior of both $ H(r) $ and $ H'(r) $ as $ r \to +\infty $, we can infer that also \eqref{perturbation} holds for suitable constants $ K>\kappa $ and $ r_0>R_0 $. Hence, Theorem \ref{thm-non-modelli} is applicable and uniqueness is thus established.
	\end{proof}

\begin{oss}\label{models-example-bis} \rm
	The purpose of Corollary \ref{thmuniq-qq} is to exhibit a significant example where our well-posedness theory works even for certain Cartan-Hadamard manifolds that are not models. As already mentioned, the case of negative quadratic curvature is particularly relevant since, at the level of powers, it corresponds to a critical growth ensuring stochastic completeness (a faster power growth yields stochastic \emph{incompleteness}, see e.g.~\cite[Introduction]{GMP-transac}). Corollary \ref{thmuniq-qq} shows that a small enough perturbation of the quadratic curvature, of order ${\left(\log r\right)^{-1}} $, is allowed. If, instead, the bound on the Ricci curvaure is of the form \eqref{ricci-general-two} for some $ \sigma \in (0,2) $, then by means of similar computations one can check that the rate of the admissible perturbation is $ r^{-\sigma} $. Clearly, more general examples could be constructed: to this purpose, it is useful to observe that the ansatz
	$$
	\phi(r) = \frac{\psi(r)}{H(r)^\kappa} \, ,
	$$
	for a suitable $ \kappa>0 $ and $ r $ large enough, often provides the correct function (up to small modifications) complying with \eqref{perturbation} and satisfying $  {\phi''}/{\phi} \le {\psi''}/{\psi}  $.
	\end{oss}

\subsection{Second strategy}\label{strategy-fabio}

The strategy we will exploit to prove Theorem \ref{thm-non-modelli-fabio} follows the lines of proof of \cite[Theorem 2.3]{GMP-pures}, and goes back to a powerful duality method introduced and developed in \cite{ACP,Pierre,BCP}. Nonetheless, some nontrivial modifications are required due to the more general curvature assumption \eqref{thmexi-ricci}.

\begin{proof}[Proof of Theorem \ref{thm-non-modelli-fabio}]
In order to simplify the discussion, we make the extra assumption  $ u , v \in C(M \times [0, {T}])$: if this is not the case, some technical modifications that rely on local cutoff arguments, which to the present purposes are inessential, should be implemented (see the beginning of the proof of \cite[Theorem 2.3]{GMP-pures} for more details). Hence, if both $u$ and $v$ satisfy \eqref{q50} and are continuous, an integration by parts (up to a further cutoff approximation) yields
\begin{equation}\label{n3}
\begin{aligned}
& \int_0^{{T}} \int_{B_R} \left[ \left(u - v\right) \xi_t + \left(u^m - v^m\right) \Delta \xi \right] d\mu dt - \int_0^{{T}} \int_{S_R} \left(u^m-v^m\right) \frac{\partial \xi}{\partial \nu} \, d\mu_{N-1} dt \\
= & \int_{B_R} \left[ u(x,{T}) - v(x,{T}) \right] \xi(x,{T}) \, d\mu(x) \, ,
\end{aligned}
\end{equation}
for all $ R>0 $ and every $\xi\in C^\infty\!\left(\overline{B}_R \times [0, {T}]\right)$ such that $ \xi\equiv 0 $ on $ S_R \times (0,{T}) $, where $\nu$ is the unit outward normal vector to $S_R$. Let us set
\begin{equation}\label{n4}
a(x,t):=
\begin{cases}
\frac{u^m(x,t) - v^m(x,t)}{u(x,t)- v(x,t)} & \text{if } u(x,t)\neq v(x,t) \, , \\
0 & \text{if } u(x,t)=v(x,t) \, .
\end{cases}
\end{equation}
In view of \eqref{e20z}, it is readily seen that
\begin{equation*}\label{n5}
0 \leq  a(x,t) \leq C_1  \, \mathcal B(r(x)) \qquad \forall (x,t) \in M \times [0, {T}]
\end{equation*}
for some constant $C_1>0$ depending on $ C $ and $ m $. By \eqref{n4}, we can rewrite \eqref{n3} as follows:
\begin{equation}\label{n6}
\begin{aligned}
& \int_0^{{T}} \int_{B_R} \left(u - v\right) \left( \xi_t + a \, \Delta \xi \right) d\mu dt - \int_0^{{T}} \int_{S_R} \left(u^m-v^m\right) \frac{\partial \xi}{\partial \nu} \, d\mu_{N-1} dt \\
= & \int_{B_R} \left[ u(x,{T}) - v(x,{T}) \right] \xi(x,{T}) \, d\mu(x) \, .
\end{aligned}
\end{equation}
For any given $R_0 \ge  2 $, let $R \ge  R_0+1 \ge 3 $. We now consider a family $ \{ a_n \} $ of smooth approximations of $ a $ (possibly depending on $R$) such that
\begin{equation}\label{n7}
\left\{a_n\right\} \subset C^\infty(M\times [0, {T}]) \, , \qquad a_n(x,t)>0 \quad \forall(x,t) \in \overline{B}_R \times [0,T] \, , \ \forall n \in \mathbb N \, .
\end{equation}
Furthermore, we can assume without loss of generality that there exist $n_0=n_0(R) \in \mathbb{N}$  and $ C_2=C_2(C_1)>0 $ such that
\begin{equation}\label{n20}
a_n(x,t) \leq C_2 \, \mathcal B(r(x)) \qquad \forall (x,t) \in B_R \times (0, {T}) \, , \ \forall n \ge n_0 \, .
\end{equation}
For the details of such a construction, we refer again to the proof of \cite[Theorem 2.3]{GMP-pures}.

Let $ \omega$ be an arbitrary smooth function on $M$ such that
\begin{equation}\label{n8}
\omega\in C^\infty_c(M) \, , \qquad 0 \leq \omega\leq 1 \quad \text{in } M \, , \qquad \omega \equiv 0 \quad \text{in } M \setminus B_{R_0} \, .
\end{equation}
For each $n\in \mathbb N$, thanks to the above assumptions and standard parabolic theory, there exists a unique classical solution $\xi_n$ to the backward problem
\begin{equation}\label{n9}
\begin{cases}
\partial_t \xi_n + a_n \, \Delta \xi_n = 0 & \text{in } B_R \times (0, {T}) \, , \\
\xi_n = 0 & \text{on } \partial B_R \times (0, {T}) \, , \\
\xi_n = \omega & \text{on } B_R \times \{ {T}\} \, .
\end{cases}
\end{equation}
If we plug $\xi \equiv \xi_n$ in \eqref{n6} and exploit \eqref{n9}, we obtain:
\begin{equation}\label{n14}
\begin{aligned}
& \int_0^{{T}} \int_{B_R} \left(u - v\right) \left( a_n - a \right)  \Delta \xi_n \, d\mu dt - \int_0^{{T}} \int_{S_R} \left(u^m-v^m\right) \frac{\partial \xi_n}{\partial \nu} \, d\mu_{N-1} dt \\
= & \int_{B_R} \left[ u(x,{T}) - v(x,{T}) \right] \omega(x) \, d\mu(x) \, .
\end{aligned}
\end{equation}
Let us label the two integral quantities in the first line of \eqref{n14}:
\begin{equation}\label{n15}
I_{n}(R) := \int_0^{{T}} \int_{B_R} \left(u - v\right) \left( a_n - a \right)  \Delta \xi_n \, d\mu dt \, ,
\end{equation}
\begin{equation*}\label{n16}
J_{n}(R) := - \int_0^{{T}} \int_{S_R} \left(u^m-v^m\right) \frac{\partial \xi_n}{\partial \nu} \, d\mu_{N-1} dt \, .
\end{equation*}
In view of condition \eqref{e20z}, we can infer that
\begin{equation}\label{n50aa}
\begin{aligned}
\left| J_n(R) \right|  \leq & \,  {T} \, \mu_{N-1}(S_R) \sup_{S_R \times (0, {T})} \left|\frac{\partial \xi_n}{\partial \nu}\right|\sup_{S_R\times(0, {T})} \left| u^m - v^m \right| \\
\leq & \, 2  C^m \left[ \mathcal B(R) \right]^{\frac m{m-1}} \mu_{N-1}(S_R) \sup_{S_R\times(0, {T})} \left|\frac{\partial \xi_n}{\partial \nu}\right| .
\end{aligned}
\end{equation}
In order to obtain a good bound for the normal derivative of $ \xi_n $ appearing in \eqref{n50aa}, first of all we set
\[
\vartheta(r) := \log[\psi(r)] \qquad \forall  r\geq R_0 \, .
\]
We then claim that the function
\begin{equation}\label{n21}
\eta(r, t):= \lambda \, \exp\left\{-\frac{K}{2 {T}- t} \, \vartheta(r)\right\} \qquad \forall(r,t)\in [R_0,+\infty) \times [0, {T}]
\end{equation}
satisfies, for all $\lambda>0$, the differential inequality
\begin{equation}\label{n33}
\eta_t + a_n \, \Delta \eta \leq 0 \qquad \text{in } \left( B_{R} \setminus \overline{B}_{R_0} \right) \times (0, {T}) \, ,
\end{equation}
provided $ K>0 $ and $ T $ fulfill suitable estimates. To prove the claim, let us observe that
\begin{equation}\label{n22}
\eta_t(r, t) = - \frac{ K}{(2{T}-t)^2} \, \vartheta(r) \, \eta(r, t) < 0 \qquad \forall(r,t)\in [R_0,+\infty) \times [0, {T}] \, .
\end{equation}
Furthermore,  thanks to \eqref{e1}, \eqref{e5} the fact that $ \psi' , \psi''\geq 0$ and the definition of $ \vartheta(r) $, we obtain:
\begin{equation}\label{f2}
\Delta \vartheta(r)=\frac{\psi''(r)}{\psi(r)}-\left[\frac{\psi'(r)}{\psi(r)}\right]^2+ \mathsf{m}(r, \theta) 	\, \frac{\psi'(r)}{\psi(r)}\geq -\left[\frac{\psi'(r)}{\psi(r)}\right]^2 \qquad \forall r\geq R_0 \, , \ \forall \theta \in \mathbb{S}^{N-1} \, .
\end{equation}
Recalling that for any function $f$ of class $C^2$ it holds
\[
\Delta\big(e^f\big)=e^f \, \Delta f + e^f \left| \nabla f \right|^2 ,
\]
we readily deduce that
\begin{equation}\label{f3}
\Delta \eta(r, t) = -\eta(r,t)  \, \frac{K}{2 T-t} \, \Delta \vartheta(r)+\eta(r,t) \left(\frac{K}{2 T-t} \right)^2 \left|\nabla \vartheta (r) \right|^2  \qquad \forall(r,t)\in [R_0,+\infty) \times [0, {T}]  \, .
\end{equation}
From \eqref{n20}, \eqref{n22}, \eqref{f2} and \eqref{f3} we thus infer that
\begin{equation*}\label{f1}
\begin{aligned}
\eta_t + a_n \, \Delta \eta = & \,  -\frac{K \eta \, \vartheta}{\left(2 T-t\right)^2} + a_n \, \eta \left[ -\frac{K}{2 T-t} \, \Delta\vartheta+\left(\frac{K}{2 T-t}\right)^2 \left|\nabla \vartheta \right|^2 \right] \\
\leq & \,   -\frac{K \eta \, \vartheta}{\left(2 T-t\right)^2} + a_n \, \eta  \left[\frac{ K}{2 T-t}\left(\frac{\psi'}{\psi}\right)^2+\left(\frac{K}{2 T-t}\right)^2\left(\frac{\psi'}{\psi}\right)^2 \right]\\
\leq & \,  \frac{K \eta}{\left(2 T-t\right)^2}\left[-\vartheta + a_n\left(2 T+K\right)\left(\frac{\psi'}{\psi}\right)^2 \right] \\
\leq & \, \frac{K \eta}{\left(2 T-t\right)^2}\left[-\vartheta + C_2  \left(2 T+K\right) \mathcal{B} \left(\frac{\psi'}{\psi}\right)^2 \right] \\
= & \, - \frac{K \eta \, \vartheta}{\left(2 T-t\right)^2} \left[  1 - C_2  \left(2 T+K\right)  \right] \qquad \text{in } \left( B_{R} \setminus \overline{B}_{R_0} \right) \times (0, {T}) \, ,
\end{aligned}
\end{equation*}
so that \eqref{n33} holds provided
\begin{equation}\label{K-T}
2T+K \le \frac{1}{C_2} \, .
\end{equation}
We now pick $\lambda$ so large that
\[
\lambda \geq \exp \left\{\frac K T  \,\vartheta(R_0) \right\} .
\]
In particular, since $ \| \omega \|_\infty \le 1 $, from \eqref{n21} it follows that
\begin{equation}\label{n33-bis}
\eta \geq \| \omega\|_\infty \qquad \text{on } \partial B_{R_0} \times (0, {T}) \, .
\end{equation}
Therefore, thanks to \eqref{n8}, \eqref{n9}, \eqref{n33}, \eqref{n33-bis} and the fact that  $ \xi_n \le \| \omega \|_\infty $ in $ B_R \times (0,T) $ (consequence of standard comparison), it turns out that $\eta$ is a supersolution to the Cauchy-Dirichlet problem
\begin{equation*}\label{n9b}
\begin{cases}
w_t  + a_n \, \Delta w = 0 & \text{in } \left(B_R \setminus \overline{B}_{R_0} \right) \times (0,{T}) \, , \\
w = 0 &  \text{on } \partial B_R \times (0, {T}) \, , \\
w = \xi_n & \text{on } \partial B_{R_0} \times (0, {T}) \, ,\\
w = 0 & \text{on } \left(B_R \setminus \overline{B}_{R_0} \right) \times \{ {T} \} \, ,
\end{cases}
\end{equation*}
of which $ \xi_n $ is by construction a solution. As a consequence, we have that $ \xi_n \le \eta  $ in $ \left(B_R \setminus \overline{B}_{R_0} \right) \times (0,{T}) $, and by arguing exactly as in the proof of \cite[Theorem 2.3]{GMP-pures} we can deduce the estimate
\begin{equation*}\label{e2lu}
\left|\frac{\partial \xi_n(x,t)}{\partial \nu}\right| \le \frac{N-2}{R} \, \frac{\left( 1-1/R \right)^{N-2}}{1-\left( 1-1/R \right)^{N-2}} \, \eta(R-1, 0) \qquad \forall (x,t) \in \partial B_R \times (0,{T})
\end{equation*}
in the case $ N \ge 3 $, and the estimate
\[
\left|\frac{\partial \xi_n(x,t)}{\partial \nu}\right| \le \frac{\eta(R-1, 0)}{R\left[\log(R)-\log(R-1)\right]} \qquad \forall (x,t) \in \partial B_R \times (0,{T})
\]
in the case $ N = 2 $. Hence, recalling \eqref{n21}, we can assert that there exists a suitable constant $ \widehat C>0 $, depending only on $N $, such that
\begin{equation*}\label{e30aa}
\left|\frac{\partial \xi_n(x,t)}{\partial \nu}\right| \le \widehat C \, \lambda \, \exp\left\{-\frac K {2{T}} \, \vartheta(R-1)\right\}  \qquad \forall (x,t) \in \partial B_R \times (0,{T}) \, .
\end{equation*}
By virtue of \eqref{thmexi-ricci} we know that the surface-comparison inequality \eqref{n50a} holds, so that going back to \eqref{n50aa} we obtain:
\begin{equation*}
\begin{aligned}
\left| J_n(R) \right| \leq \, & 2  C^m \widehat{C} \lambda \left| \mathbb{S}^{N-1} \right| \left[ \psi(R) \right]^{N-1} \left[ \psi(R-1) \right]^{-\frac{K}{2T}} \left[ \mathcal B(R) \right]^{\frac m{m-1}}  \\
 \le \, & 2  C^m \widehat{C} \lambda \left| \mathbb{S}^{N-1} \right| \left[ \psi(R) \right]^{N-1+\frac{2m}{m-1}} \left[ \psi(R-1) \right]^{-\frac{K}{2T}}  \\
 \le \, & 2  C^m \widehat{C} \lambda \left| \mathbb{S}^{N-1} \right| \left[ \psi(R) \right]^{N-1+\frac{2m}{m-1} - \frac{K}{2Tl} } ,
\end{aligned}
\end{equation*}
where we used the fact that $ \psi' \ge 1  $ and, in the last passage, assumption \eqref{log}. As a result, because $ \lim_{R \to +\infty} \psi(R)=+\infty $, it follows that
\begin{equation}\label{n52}
\limsup_{R\to +\infty} \limsup_{n\to \infty} \left| J_n(R) \right| =  0
\end{equation}
provided
\begin{equation}\label{n52-bis}
T < \frac{(m-1)K}{2l \left[ (N-1)(m-1)+2m \right]} \, .
\end{equation}
It is easy to check that conditions \eqref{K-T} and \eqref{n52-bis} are compatible and lead to an upper constraint on $T$ that only depends on the constants $ C , l, m , N $.

 As concerns the quantity $ I_n(R) $ introduced in \eqref{n15}, one can reason exactly as in the final part of the proof of \cite[Theorem 2.3]{GMP-pures} and construct the above sequence $ \{ a_n \} $ in such a way that \eqref{n7}, \eqref{n20} hold and in addition
\begin{equation}\label{n53}
\limsup_{n \to \infty} \left| I_n(R) \right| = 0 \qquad \forall R \ge R_0 + 1 \, .
\end{equation}
Indeed, the technique used there is purely local (i.e.~curvature conditions at infinity do not play a role).

Finally, if we let first $ n \to \infty $ and then $ R \to +\infty $ in \eqref{n14}, by taking advantage of \eqref{n52} and \eqref{n53} we end up with the identity
\begin{equation}\label{q11}
\int_{M} \left[ u(x,{T}) - v(x,{T}) \right] \omega(x) \, d\mu(x) = 0 \, .
\end{equation}
Because $R_0$ can be taken arbitrarily large and $ \omega $ is any regular function subject to \eqref{n8}, from \eqref{q11} we infer that $u(\cdot,{T})=v(\cdot,{T})$, and clearly the same holds for every $ t<T $ since the whole argument can be repeated upon replacing $T$ with $ t $. The only constraint we still have to remove is the smallness of ${T}$: to this end, it is enough to apply the short-time uniqueness result a finite number of times.
\end{proof}

The same technique of proof of Theorem \ref{thm-non-modelli-fabio}, with inessential modifications, entails a comparison principle for sub/supersolutions.
\begin{cor}\label{comppr}
Let $ M $ be an $ N $-dimensional Cartan-Hadamard manifold such that \eqref{thmexi-ricci} holds for some $ o \in M $, where $\psi$ is the model function associated with an $N$-dimensional Cartan-Hadamard model manifold satisfying \eqref{log} for some $ l>1 $. Given $ T>0 $, let $u$ and $ v$ be a subsolution and a supersolution, respectively, to problem \eqref{e64} corresponding to the same initial datum $ u_0 \in L^\infty_{loc}(M) $. Suppose in addition that $ u$ and $ v $ comply with the pointwise bound \eqref{e20z}. Then $ u \leq v $ a.e.~in $ M \times (0, T)$.
\end{cor}

\noindent {\bf Acknowledgments.} The authors were partially supported by the PRIN Project  ``Direct and Inverse Problems for Partial Differential Equations: Theoretical Aspects and Applications'' (grant no.~201758MTR2, MIUR, Italy). They also thank the GNAMPA group of the Istituto Nazionale di Alta Matematica (INdAM, Italy).

\end{document}